\documentclass[11pt]{amsart}
\usepackage{hyperref}
\hypersetup{nesting=true,debug=true,naturalnames=true}
\usepackage{graphicx,amssymb,upref}

\usepackage[usenames,dvipsnames]{pstricks}
\usepackage{epsfig}
\usepackage{color}
\usepackage[latin1]{inputenc}
\usepackage[T1]{fontenc}
\usepackage{amsmath, amsthm}
\usepackage[english]{babel}
\usepackage{amssymb}
\usepackage{latexsym}
\usepackage{graphicx,amssymb,upref}
\usepackage[all]{xy}
\usepackage{hyperref}
\hypersetup{nesting=true,debug=true,naturalnames=true}
\usepackage{tikz}
\usetikzlibrary{matrix}

\newtheorem{theorem}{Theorem}[section]

\newtheorem{corollary}[theorem]{Corollary}
\theoremstyle{definition}
\newtheorem{definition}[theorem]{Definition}
\newtheorem{example}[theorem]{Example}

\newtheorem{remark}[theorem]{Remark}
\newtheorem{notation}[theorem]{Notation}

\newcommand{\ot}{\otimes}
\newcommand{\co}{\circ}

\let\<\langle
\let\>\rangle

\let\uml\"

\title[Categorical equivalences for Hopf trusses and their modules]{Categorical equivalences for Hopf trusses and their modules}

\title{Categorical equivalences for Hopf trusses and their modules} 

\begin{document}

\maketitle

\begin{center}
	{\bf  Ram\'on
		Gonz\'{a}lez Rodr\'{\i}guez$^{1}$ and Ana Bel\'en Rodr\'{\i}guez Raposo$^{2}$}.
\end{center}

\vspace{0.4cm}

\begin{center}
	{\small $^{1}$ [https://orcid.org/0000-0003-3061-6685].}
\end{center}
\begin{center}
	{\small  CITMAga, 15782 Santiago de Compostela, Spain.}
\end{center}
\begin{center}
	{\small  Universidade de Vigo, Departamento de Matem\'{a}tica Aplicada II,  E. E. Telecomunicaci\'on,
		E-36310  Vigo, Spain.
		\\email: rgon@dma.uvigo.es}
\end{center}

\vspace{0.2cm}

\begin{center}
	{\small $^{2}$ [https://orcid.org/0000-0002-8719-5159]}
\end{center}
\begin{center}
	{\small  CITMAga, 15782 Santiago de Compostela, Spain.}
\end{center}
\begin{center}
	{\small Universidade de Santiago de Compostela, Departamento de Did\'acticas Aplicadas, Facultade C. C. Educaci\'on, E-15782 Santiago de Compostela, Spain.
		\\email: anabelen.rodriguez.raposo@usc.es}
\end{center}

\begin{abstract}  
In this paper we introduce the notion of generalized invertible 1-cocycle in a strict braided monoidal category ${\sf C}$, and we prove that the category of Hopf trusses  is equivalent to the category of generalized invertible 1-cocycles. On the other hand,  we also introduce the notions of module for a Hopf truss and for a generalized invertible 1-cocycle. We prove some functorial results involving these categories of modules and we show that the category of modules associated to a generalized invertible 1-cocycle is equivalent to a category of  modules associated to a suitable Hopf truss. Finally, assuming that in ${\sf C}$ we have equalizers, we introduce the notion of Hopf-module in the Hopf truss setting and we obtain the Fundamental Theorem of Hopf modules associated to a Hopf truss.
 \end{abstract} 

\vspace{0.2cm} 



\section{Introduction}  

The notion of skew brace was introduced recently in  \cite{GV}. This algebraic object consists of two different group structures $(T, \diamond)$ and $(T, \circ)$ on the same set $T$ and that satisfy $\forall a, b, c\in T$ the compatibility condition
$$a\circ (b\diamond c)=(a\circ b)\diamond a^{\diamond}\diamond (a\circ c)$$
where $a^{\diamond}$ denotes the inverse with respect to $\diamond$. Hopf braces were introduced in \cite{AGV} as the linearisation of skew braces and then consist of two structures $(H_1, H_2)$ of Hopf algebras defined on the same object that share a common coalgebra structure. As a consequence of the compatibility condition, the Hopf algebra $H_1$ can be endowed with a structure of module algebra over $H_2$. The relevance of this structure comes through bijective  1-cocycles $\sigma:H_1\to B$, where $B$ is a Hopf algebra that acts on $H_1$ (see \cite{AGV}). In fact, the category of Hopf braces with fixed $H_1$ is equivalent to the category of invertible 1-cocycles $\sigma:H_1\to B$ (see \cite[Theorem 1.12]{AGV}). Thus, Hopf braces are nothing more than coalgebra isomorphisms between Hopf algebras that share the underlying coalgebra and related by a module algebra structure. 

In \cite{BRZ1} T. Brzezi\'nski introduces the notion of skew truss as a generalization of the notion of skew brace in the following way: A skew truss  consist of a group structure $(T, \diamond)$ and a semigroup structure $(T, \circ)$ defined on the same set $T$, and a map (called the cocycle of the skew truss) $\omega:T\to T$ that satisfy $\forall a, b, c\in T$:
$$a\circ (b\diamond c)=(a\circ b)\diamond \omega (a)^{\diamond}\diamond (a\circ c)$$
where $\omega (a)^{\diamond}$ is the inverse in $(T, \diamond)$. This object can be linearised to obtain the so-called Hopf truss $(H_1, H_2, \sigma)$ that consists of a Hopf algebra $H_1$ and a non-unital bialgebra $H_2$ with the same underlying coalgebra structure, and a cocycle $\sigma:H_2\to H_1$ that must satisfy the corresponding compatibility condition. Due to it, $H_1$ can be endowed with a non-unital $H_2$-module algebra structure. Observe that if $H_2$ is a Hopf algebra and $\sigma$ is the identity, the Hopf truss is actually a Hopf brace. In the second section of this  paper we prove that there exists a relation between the categories of Hopf trusses and  skew trusses in the following way (see Theorem \ref{111}):  The category of skew trusses is equivalent to  the full subcategory of Hopf trusses whose  objects are  pointed cosemisimple   Hopf trusses.

Taking into account the equivalences between the categories of Hopf braces an invertible 1-cocycles and with the intention of extending them to the context of Hopf trusses, in the third section of this paper we introduce the notion of generalized invertible 1-cocycle in a strict braided monoidal category ${\sf C}$ as a generalization of the bijective 1-cocycles proposed by I. Angiono, C. Galindo and L. Vendramin in \cite{AGV}. We subsequenty prove in Theorem  \ref{EGIHT}  that the category of Hopf trusses  is equivalent to the category of generalized invertible 1-cocycles. In Section 4 we introduce the notions of module for a Hopf truss and for a generalized invertible 1-cocycle and we prove some functorial results involving these categories of modules. The main result of section 4 is Theorem \ref{pHpi}, where we show that the category of modules associated to a generalized invertible 1-cocycle is equivalent to the category of  modules associated to the corresponding Hopf truss obtained in the previous section. This results are a generalization to the Hopf truss setting of the ones proved for modules associates to Hopf braces in \cite{VRBAMod}.

In the last part of this paper we explore the so-called Fundamental Theorem of Hopf Braces. It is well known that if $H$ is a Hopf algebra in a category ${\mathbb F}$-{\sf Vect} of vector spaces over a field ${\mathbb F}$ and 
 $X$ is an object in this category, then the tensor product $H\ot X$, with the action and coaction induced by the product and the coproduct of $H$, is an object in the category ${\sf H}$-{\sf Hopf-Mod}, i.e. the category of left $H$-Hopf modules. The objects and morphisms of this category are defined in the following way: Let $M$ be a left $H$-module and a left $H$-comodule. In this setting, if for all $m\in M$ and $h\in H$, we write $h.m$ for the left action and  we use the Heyenman-Sweedler notation $\rho_{M}(m)=m_{[0]}\ot m_{[1]}$ for the coaction, then we will say that $M$ is a left $H$-Hopf module if the equality  
$$\rho_{M}(h.m)=h_{(1)}m_{[0]}\ot h_{(2)}.m_{[1]}$$
holds, where $\delta_{H}(h)=h_{(1)}\ot h_{(2)}$ is the coproduct of $H$ and $h_{(1)}m_{[0]}$ is the product in $H$ of $h_{(1)}$ and $m_{[0]}$. A morphism between two left $H$-Hopf modules is a ${\Bbb F}$-linear map that is $H$-linear and $H$-colinear. 

This construction introduced for $H\otimes X$ in the previous paragraph is functorial and, as a consequence,  we have a functor, called the induction functor, $F=H\ot -:{\mathbb F}$-{\sf Vect}$\;\rightarrow {\sf H}$-{\sf Hopf-Mod}. Moreover, for all $M\in H$-{\sf Hopf}, the construction of subobject of coinvariants $M^{co H}=\{m\in M\;/\; m_{[0]}\ot m_{[1]} =1_{H}\otimes m\}$ also is functorial. Thus, there exists a functor of coinvariants $G=(\;\;)^{co H}: {\sf H}$-{\sf Hopf-Mod}$\;\rightarrow {\mathbb F}$-{\sf Vect} such that $F\dashv G$. Moreover, $H\ot M^{co H}$ and $M$ are isomorphic in  ${\sf H}$-{\sf Hopf-Mod} (see \cite{Larson-Sweedler} and \cite{Sweedler}) and $F$ and $G$ induces an  equivalence between  the categories ${\sf H}$-{\sf Hopf-Mod} and  ${\mathbb F}$-{\sf Vect}. The existence of the isomorphism between $H\ot M^{co H}$ and $M$  is the main statement of the Fundamental Theorem of Hopf modules.  

The Fundamental Theorem of Hopf modules and the categorical equivalence of the previous paragraph  remain valid for weak Hopf algebras, Hopf quasigroups, weak Hopf quasigroups and in \cite{RGON} we can find that which can also be obtained for Hopf braces. In the last section of this paper, assuming that the braided monoidal category ${\sf C}$ admits equalizers, we extend the previous results to the Hopf truss setting, i.e. we introduce the notion of Hopf-module associated to a Hopf truss, we obtain the Fundamental Theorem of Hopf modules associated to a Hopf truss and we prove that there exists a categorical equivalence as in the case of Hopf algebras, weak Hopf algebras, Hopf quasigroups, weak Hopf quasigroups, Hopf braces, etc.

\section{Hopf trussses} 

Throughout this paper ${\sf  C}$ denotes a strict braided monoidal category with tensor product $\ot$, unit object $K$ and braiding $c$. Recall that a monoidal category is a category ${\sf  C}$ together with a functor $\ot :{\sf  C}\times {\sf  C}\rightarrow {\sf  C}$, called tensor product, an object $K$ of ${\sf C}$, called the unit object, and  families of natural isomorphisms 
$$a_{M,N,P}:(M\ot N)\ot P\rightarrow M\ot (N\ot P),\;\;\;r_{M}:M\ot K\rightarrow M, \;\;\; l_{M}:K\ot M\rightarrow M,$$
in ${\sf  C}$, called  associativity, right unit and left unit constraints, respectively, satisfying the Pentagon Axiom and the Triangle Axiom, i.e.,
$$a_{M,N, P\ot Q}\co a_{M\ot N,P,Q}= (id_{M}\ot a_{N,P,Q})\co a_{M,N\ot P,Q}\co (a_{M,N,P}\ot id_{Q}),$$
$$(id_{M}\ot l_{N})\co a_{M,K,N}=r_{M}\ot id_{N},$$
where for each object $X$ in ${\sf  C}$, $id_{X}$ denotes the identity morphism of $X$ (see  \cite{Mac}). A monoidal category is called strict if the constraints of the previous paragraph are identities. It is a well-known  fact (see for example \cite{Christian}) that every non-strict monoidal category is monoidal equivalent to a strict one. This lets us to treat monoidal categories as if they were strict and, as a consequence, the results proved in a strict setting hold for every non-strict  monoidal category, for example the category ${\mathbb F}$-{\sf Vect} of vector spaces over a field ${\mathbb F}$,  the category $R$-{\sf Mod} of left modules over a commutative ring $R$ , or {\sf Set} the category of sets.

For simplicity of notation, given objects $M$, $N$, $P$ in ${\sf  C}$ and a morphism $f:M\rightarrow N$,  we will write $P\ot f$ for $id_{P}\ot f$ and $f \ot P$ for $f\ot id_{P}$.

A braiding for a strict monoidal category ${\sf  C}$ is a natural family of isomorphisms 
$$c_{M,N}:M\ot N\rightarrow N\ot M$$ subject to the conditions 
$$
c_{M,N\ot P}= (N\ot c_{M,P})\co (c_{M,N}\ot P),\;\;
c_{M\ot N, P}= (c_{M,P}\ot N)\co (M\ot c_{N,P}).
$$

A strict braided monoidal category ${\sf  C}$ is a strict monoidal category with a braiding.  Note that, as a consequence of the definition, the equalities $c_{M,K}=c_{K,M}=id_{M}$ hold, for all object  $M$ of ${\sf  C}$.  If the braiding satisfies that  $c_{N,M}\co c_{M,N}=id_{M\ot N},$ for all $M$, $N$ in ${\sf  C}$, then we will say that ${\sf C}$  is symmetric. In this case, we call the braiding $c$ a symmetry for the category ${\sf  C}$.

\begin{definition}
A monoid in ${\sf  C}$ is a triple $A=(A, \eta_{A}, \mu_{A})$ where $A$ is an object in ${\sf  C}$ and
$\eta_{A}:K\rightarrow A$ (unit), $\mu_{A}:A\otimes A
\rightarrow A$ (product) are morphisms in ${\sf  C}$ such that
$\mu_{A}\circ (A\otimes \eta_{A})=id_{A}=\mu_{A}\circ
(\eta_{A}\otimes A)$ (unit property) and  $\mu_{A}\circ (A\otimes \mu_{A})=\mu_{A}\circ
(\mu_{A}\otimes A)$ (associative property) hold. 

Given two monoids $A= (A, \eta_{A}, \mu_{A})$
and $B=(B, \eta_{B}, \mu_{B})$, a morphism  $f:A\rightarrow B$ in {\sf  C} is a monoid morphism if $\mu_{B}\circ (f\otimes f)=f\circ \mu_{A}$ ($f$ is multiplicative) and $ f\circ
\eta_{A}= \eta_{B}$ ($f$ preserves the unit). 
		
If  $A$, $B$ are monoids in ${\sf  C}$, then the tensor product
$A\otimes B$ is also an algebra in
${\sf  C}$ where
$\eta_{A\otimes B}=\eta_{A}\otimes \eta_{B}$ and $\mu_{A\otimes B}=(\mu_{A}\otimes \mu_{B})\circ (A\otimes c_{B,A}\otimes B).$

We will say that a monoid $A$ is commutative if $\mu_{A}\circ c_{A,A}=\mu_{A}$.
\end{definition}

\begin{definition}
A comonoid  in ${\sf  C}$ is a triple ${D} = (D,
\varepsilon_{D}, \delta_{D})$ where $D$ is an object in ${\sf
C}$ and $\varepsilon_{D}: D\rightarrow K$ (counit),
$\delta_{D}:D\rightarrow D\otimes D$ (coproduct) are morphisms in
${\sf  C}$ such that $(\varepsilon_{D}\otimes D)\circ
\delta_{D}= id_{D}=(D\otimes \varepsilon_{D})\circ \delta_{D}$ (counit property) and 
$(\delta_{D}\otimes D)\circ \delta_{D}=(D\otimes \delta_{D})\circ \delta_{D}$ (coassociative property) hold.

If ${D} = (D, \varepsilon_{D},
\delta_{D})$ and
${ E} = (E, \varepsilon_{E}, \delta_{E})$ are comonoids, then a morphism 
$f:D\rightarrow E$ in  {\sf  C} is a comonoid morphism if $(f\otimes f)\circ
\delta_{D} =\delta_{E}\circ f$ ($f$ is comultiplicative) and $\varepsilon_{E}\circ f
=\varepsilon_{D}$ ($f$ preserves the counit).
		
Given  $D$, $E$ comonoids in ${\sf  C}$, the tensor product $D\otimes E$ is a comonoid in ${\sf  C}$ where $\varepsilon_{D\otimes
E}=\varepsilon_{D}\otimes \varepsilon_{E}$ and $\delta_{D\otimes
E}=(D\otimes c_{D,E}\otimes E)\circ( \delta_{D}\otimes \delta_{E}).$

We will say that a comonoid $D$ is cocommutative if $\delta_{D}=c_{D,D}\circ \delta_{D}.$
\end{definition}

\begin{definition}
Let $D=(D,\varepsilon_{D}, \delta_{D})$ be a comonoid in {\sf C}. We will say that a morphism $g:K\rightarrow D$ is a grouplike morphism if satisfy
\begin{equation}
\label{gl1}
\delta_{D}\circ g=g\otimes  g
\end{equation}
and 
\begin{equation}
\label{gl2}
\varepsilon_{D}\circ g=id_{K}.
\end{equation}

\end{definition}

\begin{remark}
In the category of vector spaces over a field ${\mathbb F}$ we can find interesting examples of comonoids. For example, if $S$ is a set, then with ${\mathbb F}[S]$ we will denote the free ${\mathbb F}$-vector space on $S$, i.e.,  
$${\mathbb F}[S]=\displaystyle \bigoplus_{s\in S}{\mathbb F}s.$$

This vector space has a comonoid structure determined by 
\begin{equation}
\label{freeS}
\delta_{{\mathbb F}[S]}(s)=s\otimes  s, \;\;\;\; \varepsilon_{{\mathbb F}[S]}(s)=1_{\mathbb F}.
\end{equation}

Let $(C, \varepsilon_{C}, \delta_{C})$ be a comonoid in ${\mathbb F}$-{\sf Vect}. A grouplike element $c$ of $C$ is an element  $c\in C$ such that the linear map $g_{c}:{\mathbb F}\rightarrow C$ defined by $g_{c}(1_{{\mathbb F}})=c$ is a grouplike morphism in ${\mathbb F}$-{\sf Vect}. Therefore, $c\in C$ is a grouplike element if  (\ref{freeS}) holds for $c$, $\varepsilon_{C}$ and $\delta_{C}$. In the following we will denote by ${\sf G}(C)$ the set of grouplike elements of $C$ and it is well known that they are linearly independent \cite[Theorem 2.1.2]{Abe}. If $S$ is a set, then the comonoid ${\mathbb F}[S]$ is called the grouplike comonoid of $S$ and satisfies ${\sf G}({\mathbb F}[S])=S$. Moreover the grouplike comonoid of ${\sf G}(C)$ is a subcomonoid of $C$ ($D$ is a subcomonoid of $C$ if $\delta_{C}(D)\subset D\otimes D$ or, in other words, $D$ is a comonoid with the restriction of the coproduct $\delta_{C}$ and the counit $\varepsilon_{C}$). 

A pointed comonoid in ${\mathbb F}$-{\sf Vect} is a comonoid $C$ whose simple subcomonoids are one-dimensional. Then, $C$ is pointed if and only if its coradical $C_{0}$ (the sum of the simple subcomonoids of $C$)   is the grouplike comonoid of ${\sf G}(C)$, i.e., $C_{0}={\mathbb F}[{\sf G}(C)]$. We will say that the comonoid $C$ is cosemisimple if $C=C_{0}$. Therefore, if $C$ is pointed cosemisimple,  then $C={\mathbb F}[{\sf G}(C)]$. On the other hand, if $G$ is a group and $C={\mathbb F}[G]$, then  we have that $C$ is pointed and cosemisimple. Finally, if ${\mathbb F}$ is algebraically closed and $C$ is cocommutative, then $C$ is pointed.
\end{remark}

\begin{definition}
Let ${D} = (D, \varepsilon_{D},
\delta_{D})$ be a comonoid and let $A=(A, \eta_{A}, \mu_{A})$ be a
monoid. By ${\mathcal  H}(D,A)$ we denote the set of morphisms
$f:D\rightarrow A$ in ${\sf  C}$. With the convolution operation
$$f\ast g= \mu_{A}\circ (f\otimes g)\circ \delta_{D},$$
${\mathcal  H}(D,A)$ is a monoid where the unit element is $\eta_{A}\circ \varepsilon_{D}=\varepsilon_{D}\otimes \eta_{A}$. We will say that $f:D\rightarrow A$ is convolution invertible if there exists $f^{-1}:D\to A$ such that $f \ast f^{-1}=f^{-1}\ast f=\varepsilon\ot \eta$. 
\end{definition}

\begin{definition}
Let  $A$ be a monoid. The pair
$(M,\varphi_{M})$ is a left $A$-module if $M$ is an object in
${\sf  C}$ and $\varphi_{M}:A\otimes M\rightarrow M$ is a morphism
in ${\sf  C}$ satisfying $\varphi_{M}\circ(
\eta_{A}\ot M)=id_{M}$ and 
$$\varphi_{M}\circ (A\ot \varphi_{M})=\varphi_{M}\circ
(\mu_{A}\ot M).$$

Given two left ${A}$-modules $(M,\varphi_{M})$
and $(N,\varphi_{N})$, $f:M\rightarrow N$ is a morphism of left
${A}$-modules if $\varphi_{N}\circ (A\ot f)=f\circ \varphi_{M}$ (left $A$-linearity).  
	
Then left $A$-modules with morphisms of left $A$-modules form a category that we will denote by $\;_{\sf A}${\sf Mod}.

Let $B$ an object in {\sf C}  such that there exists  an associative product $\mu_{B}:B\otimes B\rightarrow B$.   We  will say that $(M,\phi_{M})$ is a non-unital left $B$-module if  $\phi_{M}\circ (B\ot \phi_{M})=\phi_{M}\circ
(\mu_{B}\ot M)$. A morphism between non-unital left $B$-modules is a left $B$-linear morphism as in the case of morphisms for modules over a monoid. Then non-unital left $B$-modules form a category that we will denote by $\;_{\sf B}${\sf mod}.
	
\end{definition}

\begin{definition}
\label{nbimod}
A non-unital bimonoid in the category $\sf C$ is a comonoid $(B,\varepsilon_{B},\delta_{B})$ with an associative product $\mu_{B}:B\otimes B\rightarrow B$ such that $\mu_{B}$ is a comonoid morphism. Then the following identities hold:
\begin{equation}
\label{et-mu-ep}
\varepsilon_{B}\co \mu_{B}=\varepsilon_{B}\ot \varepsilon_{B},
\end{equation}
\begin{equation}
\label{d-delta-et}
\delta_{B}\co \mu_{B}=(\mu_{B}\ot \mu_{B})\co \delta_{B\otimes B}. 
\end{equation}

A bimonoid in ${\sf C}$ is a monoid $(B, \eta_{B}, \mu_{B})$ and a comonoid $(B,\varepsilon_{B},\delta_{B})$ such that $\eta_{B}$ and $\mu_{B}$ are comonoid morphisms. Then, (\ref{et-mu-ep}),  (\ref{d-delta-et}), 
\begin{equation}
\label{et-mu-ep1}
\varepsilon_{B}\co \eta_{B}=id_{K},
\end{equation}
and 
\begin{equation}
\label{d-delta-et1}
\delta_{B}\co \eta_{B}=\eta_{B}\ot \eta_{B}
\end{equation}
hold.
 		
A morphism between non-unital bimonoids $H$ and $B$ is a morphism $f:H\rightarrow B$ in ${\sf C}$ of comonoids and multiplicative.  A morphism between bimonoids $H$ and $B$ is a morphism $f:H\rightarrow B$ in ${\sf C}$ of monoids and comonoids.

With the composition of morphisms in {\sf C} we can define a category whose objects are non-unital bimonoids (bimonoids)  and whose morphisms are morphisms of non-unital bimonoids (bimonoids). We denote this category by ${\sf  bimod}$ (${\sf  BiMod}$).

\end{definition}

\begin{definition}
Let $H$ be a bimonoid in ${\sf C}$.  If there exists a morphism $\lambda_{H}:H\rightarrow H$ in ${\sf  C}$,
called the antipode of $H$, satisfying that $\lambda_{H}$ is the inverse of $id_{H}$ in ${\mathcal  H}(H,H)$, i.e., 
\begin{equation}
\label{antipode}
id_{H}\ast \lambda_{H}= \eta_{H}\circ \varepsilon_{H}= \lambda_{H}\ast id_{H},
\end{equation}
we say that $H$ is a Hopf monoid. 

A morphism of Hopf monoids is an bimonoid morphism. With the composition of morphisms in {\sf C} we can define a category whose objects are  Hopf monoids  and whose morphisms are morphisms of Hopf monoids. We denote this category by ${\sf  Hopf}$.
\end{definition}

\begin{remark}
If $H$ is a Hopf monoid,  then the antipode is antimultiplicative and anticomultiplicative 
$$
\label{a-antip}
\lambda_{H}\co \mu_{H}=  \mu_{H}\co (\lambda_{H}\ot \lambda_{H})\co c_{H,H},\;\;\;\; \delta_{H}\co \lambda_{H}=c_{H,H}\co (\lambda_{H}\ot \lambda_{H})\co \delta_{H}, 
$$
and leaves the unit and counit invariant, i.e., 
$$
\label{u-antip}
\lambda_{H}\co \eta_{H}=  \eta_{H},\;\; \varepsilon_{H}\co \lambda_{H}=\varepsilon_{H}.
$$
A Hopf monoid is commutative if it is commutative as monoid and cocommutative if it is cocommutative as comonoid. It is easy to see that in both cases $\lambda_{H}\circ \lambda_{H} =id_{H}$.

Note that, if $f:H\rightarrow D$ is a Hopf monoid morphism, then  the following equality holds:
\begin{equation}
\label{morant}
\lambda_{D}\co f=f\co \lambda_{H}.
\end{equation}
\end{remark}

\begin{remark}
\label{t-exe}
In the category {\sf Set} the tensor product is the cartesian product and the unit element is a set $\{{\star}\}$ with an unique element $\star$. 
Taking all this into account, we have that a set  $T$ is a non-unital bimonoid in {\sf Set} if and only if  $T$ is a semigroup in {\sf Set} with product $\diamond$. If $(T,\diamond)$ is a semigroup, then the non-unital bimonoid structure is the following: the product is the one induced by $\diamond$, the coproduct is defined by $\delta_{T}(a)=(a,a)$ and the counit by $\varepsilon_{T}(a)=\star$. Moreover, $T$ is a bimonoid in {\sf Set} if and only if $T$ is a monoid in {\sf Set} with product $\diamond$ and unit $1_{\diamond}$. In this case  we define the non-unital bimonoid structure as in the case of semigroups and the unit is defined by $\eta_{T}(\star)=1_{\diamond}$. Finally, $T$ is a Hopf monoid in  {\sf Set} if and only if it is a group. The bimonoid structure is defined as for monoids and the antipode by $\lambda_{T}(a)=a^{\diamond}$ where  $a^{\diamond}$ is the inverse of $a$ in $T$.

On the other hand, if ${\mathbb F}$ is a field and $(T, \diamond)$ is a semigroup in {\sf Set}, then the direct sum $${\mathbb F}[T]=\displaystyle \bigoplus_{a\in T}{\mathbb F}a$$ 
is a non-unital bimonoid in ${\mathbb F}$-{\sf Vect}  where  $\mu_{{\mathbb F}[T]}$ is the unique linear map such that  $\mu_{{\mathbb F}[T]} (a\otimes b)=a\diamond b$ and the comonoid structure is the one defined in (\ref{freeS}). Also, if $(T, \diamond)$ is a monoid in {\sf Set} with unit $1_{\diamond}$, then ${\mathbb F}[T]$ 
is a bimonoid in ${\mathbb F}$-{\sf Vect},  where $\eta_{{\mathbb F}[T]}$ is the   unique linear map such that $\eta_{{\mathbb F}[T]} (1_{\mathbb F})=1_{\diamond}$, and the non-unital bimonoid structure is the one introduced for semigroups. Finally, if $(T, \diamond)$ is a group, then ${\mathbb F}[T]$ is a Hopf monoid  in ${\mathbb F}$-{\sf Vect} with the previous bimonoid structure and antipode  the unique linear map satisfying $\lambda_{{\mathbb F}[T]}(a)=a^{\diamond}$, where $a^{\diamond}$ denotes the inverse of $a$. 

Let $(T, \diamond)$, $(S,\circ )$ be semigroups in {\sf Set} and let $f:T\rightarrow S$ be a semigroup morphism in {\sf Set}. Then, if ${\mathbb F}[f]$ denotes the linear extension of $f$, then ${\mathbb F}[f]$ is a non-unital bimonoid morphism between 
${\mathbb F}[T]$ and ${\mathbb F}[S]$ in ${\mathbb F}$-{\sf Vect}. The same property holds for a morphisms of monoids $f$ in {\sf Set}, i.e. ${\mathbb F}[f]$ is a bimonoid morphism,  and for a morphism $f$ of groups, i.e. ${\mathbb F}[f]$ is a group morphism.

Therefore, if {\sf sGpr} is the category of semigroups in {\sf Set}, {\sf Mon} is the category of monoids in {\sf Set}  and {\sf Gpr} denotes the category of groups, then there exists three functors $${\sf L}_{sg}: {\sf sGpr}\;\rightarrow {\sf bimon}, \;\;\;\;\; {\sf L}_{m} : {\sf Mon}\;\rightarrow {\sf BiMon}, \;\;\;\;\; {\sf L}_{g} : {\sf Gpr}\;\rightarrow {\sf Hopf},$$ where  {\sf bimon} is the category of non-unital bimonoids in ${\mathbb F}$-{\sf Vect}, {\sf BiMon} is the category of  bimonoids in ${\mathbb F}$-{\sf Vect} and, finally, {\sf Hopf} denotes the category of Hopf monoids in ${\mathbb F}$-{\sf Vect}. In this setting, the functor {\sf L}$_{m}$ is the restriction of {\sf L}$_{sg}$ to the category of monoids in ${\sf Set}$ and {\sf L}$_{g}$ is the restriction of {\sf L}$_{m}$ to the category of groups.

In any case non-unital bimonoids, bimonoids and Hopf monoids in ${\mathbb F}$-{\sf Vect} are comonoids. Taking this into account,  if $B$ is an object in ${\sf bimon}$, then it is possible to define a semigroup structure on ${\sf G}(B)$, with product $\diamond$ induced by $\mu_{B}$ ($a\diamond b = \mu_{B}(a\otimes b)$). Moreover, if $f:B\to B'$ is a morphism of non-unital bimonoids, then the  image of the restriction of $f$ to ${\sf G}(B)$ lies into ${\sf G}(B')$. Thus we have a functor ${\sf G}_{sg}$ between ${\sf bimon}$ and ${\sf sGpr}$ defined by ${\sf G}_{sg}(B)={\sf G}(B)$ on objects and by ${\sf G}_{sg}(f)={\sf G}(B)$ on morphisms. Also, if $A$ is an object in ${\sf BiMon}$, then it is possible to define a monoid structure on ${\sf G}(A)$, with product $\diamond$ as in the case of non-unital bimonoids and unit $1_{\diamond}$ induced by $\eta_{A}$ ($1_{\diamond} = \eta_{A}(1_{\mathbb F})$). Moreover, if $f:A\to A'$ is a morphism of bimonoids, then the  image of the restriction of $f$ to ${\sf G}(A)$ lies into ${\sf G}(A')$. Then, we have a new  functor ${\sf G}_{m}$ between ${\sf BiMon}$ and ${\sf Mon}$ defined by ${\sf G}_{m}(A)={\sf G}(A)$ on objects and by ${\sf G}_{m}(f)={\sf G}(f)$ on morphisms. Finally, If $H$ is a Hopf monoid, then ${\sf G}(H)$ is a group where the monoid structure is the one defined for bimonoids and the inverse is defined by the antipode, i.e., the inverse of $h\in {\sf G}(H)$ is $\lambda_{H}(h)$. As in the two previous cases, this construction works well with morphisms of Hopf monoids and, as a consequence, there exists a  functor ${\sf G}_{g}$ between ${\sf Hopf}$ and ${\sf Gpr}$ defined by ${\sf G}_{g}(H)={\sf G}(H)$ on objects and by ${\sf G}_{g}(f)={\sf G}(f)$ on morphisms. 

It is true that ${\sf L}_{sg}\dashv {\sf G}_{sg}$, ${\sf L}_{m}\dashv {\sf G}_{m}$ and ${\sf L}_{g}\dashv {\sf G}_{g}$. In the three adjunctions the unit of every one of them is the identity. Thus, the first  adjoint pair induces an equivalence of categories between {\sf sGpr} and the full subcategory of  ${\sf bimon}$ of all pointed cosemisimple  non-unital bimonoids in ${\mathbb F}$-{\sf Vect}, the second one induces an equivalence of categories between {\sf Mon} and the full subcategory of  ${\sf BiMon}$ of all pointed cosemisimple  bimonoids in ${\mathbb F}$-{\sf Vect}, and the third one induces an equivalence of categories between {\sf Gpr} and the full subcategory of  ${\sf Hopf}$ of all pointed cosemisimple  non-unital Hopf monoids in ${\mathbb F}$-{\sf Vect}. Also we have this commutative diagrams 

$$
\setlength{\unitlength}{4mm}
\begin{picture}(35,7.5)

\put(3,6.5){\vector(1,0){6}}
\put(9,5.5){\vector(-1,0){6}}
\put(2,5){\vector(0,-1){2}}
\put(11.5,5){\vector(0,-1){2}}
\put(3.2,2.5){\vector(1,0){6}}
\put(9.2,1.5){\vector(-1,0){6}}

\put(23,6.5){\vector(1,0){6}}
\put(29,5.5){\vector(-1,0){6}}
\put(22,5){\vector(0,-1){2}}
\put(31.5,5){\vector(0,-1){2}}
\put(23.2,2.5){\vector(1,0){6}}
\put(29.2,1.5){\vector(-1,0){6}}

\put(6.5,7.2){\makebox(0,0){$ {\sf L}_{m}$}}
\put(6.5,4.8){\makebox(0,0){$ {\sf G}_{m}$}}
\put(6.5,6){\makebox(0,0){$\perp$}} 
\put(2,6){\makebox(0,0){${\sf Mon}$ }}
\put(11.75,6){\makebox(0,0){${\sf BiMon}$}}

\put(1,4.2){\makebox(0,0){${\sf I_{m}}$}}
\put(13,4.2){\makebox(0,0){${\sf I_{b}}$}}

\put(2,2){\makebox(0,0){${\sf sGrp}$}}
\put(11.5,2){\makebox(0,0){${\sf bimon}$}}
\put(6.5,0.8){\makebox(0,0){${\sf G}_{sg}$}}
\put(6.5,2){\makebox(0,0){$\perp$}} 
\put(6.5,3.1){\makebox(0,0){${\sf L}_{sg}$}}

\put(26.5,7.2){\makebox(0,0){$ {\sf L}_{g}$}}
\put(26.5,4.8){\makebox(0,0){$ {\sf G}_{g}$}}
\put(26.5,6){\makebox(0,0){$\perp$}} 
\put(22,6){\makebox(0,0){${\sf Grp}$ }}
\put(31.75,6){\makebox(0,0){${\sf Hopf}$}}

\put(21,4.2){\makebox(0,0){${\sf I_{g}}$}}
\put(33,4.2){\makebox(0,0){${\sf I_{h}}$}}

\put(22,2){\makebox(0,0){${\sf Mon}$}}
\put(31.5,2){\makebox(0,0){${\sf BiMon}$}}
\put(26.5,0.8){\makebox(0,0){${\sf G}_{m}$}}
\put(26.5,2){\makebox(0,0){$\perp$}} 
\put(26.5,3.1){\makebox(0,0){${\sf L}_m$}}

\end{picture}
$$

where ${\sf I_{m}}$ and ${\sf I_{b}}$ and ${\sf I_{g}}$ denote the corresponding inclusion functors.

\end{remark}

\begin{definition}
\label{modmon}
Let $B$ a  bimonoid and let $A$ be a monoid in {\sf C}. We will say that $(A, \phi_{A})$ is a left $B$-module monoid if it is a  left $B$-module with action $ \varphi_{A}:B\otimes A\rightarrow A$ such that 
\begin{equation}
\label{bmm1}
\varphi_{A}\co (B\ot \eta_{A})=\varepsilon_{B}\otimes \eta_{A}
\end{equation}
and
\begin{equation}
\label{bmm2}	 
\varphi_{A}\co (B\otimes \mu_{A})= \mu_{A}\circ (\varphi_{A}\otimes \varphi_{A})\circ (B\otimes c_{B,A}\otimes A)\circ (\delta_{B}\otimes A\otimes A)
\end{equation} 
hold.
	
If $B$ is a non-unital bimonoid, then we will say that  $(A, \phi_{A})$ is a  non-unital left $B$-module monoid if   $(A, \phi_{A})$ is a non-unital left $B$-module and (\ref{bmm1}) and (\ref{bmm2}) hold.
\end{definition}

The notion of Hopf truss was introduced by T. Brzezi\'nski in \cite{BRZ1} in the category ${\mathbb F}$-{\sf Vect} as the linearisation of the notion of skew truss. In the monoidal setting the definition of Hopf truss is the following:

\begin{definition}
\label{H-truss}
 Let $(H, \varepsilon_{H}, \delta_{H})$ be a comonoid in {\sf C}. Assume that  there are a monoid structure $(H, \eta_{H}, \mu_{H}^1)$, a product $\mu_{H}^2:H\otimes H\rightarrow H$ and  two endomorphism of $H$ denoted by $\lambda_{H}$ and $\sigma_{H}$. We will say that 
$$(H, \eta_{H}, \mu_{H}^1, \mu_{H}^2, \varepsilon_{H}, \delta_{H}, \lambda_{H}, \sigma_{H})$$
is a Hopf truss if:
\begin{itemize}
\item[(i)] $H_{1}=(H, \eta_{H}, \mu_{H}^1,  \varepsilon_{H}, \delta_{H}, \lambda_{H})$ is a Hopf monoid in {\sf C}.
\item[(ii)] $H_{2}=(H,  \mu_{H}^2, \varepsilon_{H}, \delta_{H})$ is a non-unital bimonoid in {\sf C}.
\item[(iii)] The morphism $\sigma_{H}$ is a comonoid morphism and the following equality holds:
$$\mu_{H}^2\co (H\ot \mu_{H}^1)=\mu_{H}^1\co (\mu_{H}^2\ot \Gamma_{H_{1}}^{\sigma_{H}} )\co (H\ot c_{H,H}\ot H)\co (\delta_{H}\ot H\ot H),$$
\end{itemize}
where  $$\Gamma_{H_{1}}^{\sigma_{H}}=\mu_{H}^1\co ((\lambda_{H}\co \sigma_{H})\ot \mu_{H}^2)\co (\delta_{H}\ot H).$$

We will say that a Hopf truss is cocommutative if the comonoid $(H, \varepsilon_{H}, \delta_{H})$ is cocommutative.

Note that, a Hopf truss is a Hopf brace in the sense of I. Angiono, C. Galindo and L. Vendramin (see \cite{AGV}) if $\sigma_{H}$ is the identity and there exists a morphism $S_{H}:H\rightarrow H$ such that $H_{2}=(H, \eta_{H}, \mu_{H}^2, \varepsilon_{H}, \delta_{H}, S_{H})$ is a Hopf monoid. 

\end{definition}

\begin{notation}
{\rm Given a Hopf truss, we will denote it by ${\mathbb H}=(H_{1}, H_{2},  \sigma_{H})$. The morphism $\sigma_{H}$ is called the cocycle of ${\mathbb H}$.  Also, we will denote a Hopf brace by ${\mathbb H}=(H_{1}, H_{2})$.
}
\end{notation}

The proofs that we can find in Section 6 of \cite{BRZ1} can be replicated in the  braided monoidal setting since they do not depend on the symmetry of the category ${\mathbb F}$-{\sf Vect}. Then we have the following properties: By  \cite[Lemma 6.2 ]{BRZ1}  the cocycle  $\sigma_{H}$ of a Hopf truss  ${\mathbb H}$ in ${\sf C}$ is fully determined by $\eta_{H}$ and the product $\mu_{H}^2$ in the following way:
\begin{equation}
\label{cocycle}
\sigma_{H}=\mu_{H}^2\circ (H\otimes  \eta_{H}).
\end{equation}

Then, as a consequence of the associativity for the product $\mu_{H}^2$, we have that 
\begin{equation}
\label{cocycle1}
\sigma_{H}\circ \mu_{H}^2=\mu_{H}^2\circ (H\otimes  \sigma_{H})
\end{equation}
holds. Finally, by \cite[Theorem 6.5]{BRZ1} we know that the monoid $H_{1}$ is a non-unital left $H_{2}$-module monoid for the action $\Gamma_{H_{1}}^{\sigma_{H}}.$ 

\begin{definition}
\label{mor}
{\rm  Given two Hopf trusses ${\mathbb H}$  and  ${\mathbb B}$, a morphism $f$ between the two underlying objects is called a morphism of Hopf trusses if  $f:H_{1}\rightarrow B_{1}$ is a Hopf monoid morphism and $f:H_{2}\rightarrow B_{2}$ is a morphism of non-unital bimonoids. Then by \cite[Proposition 6.8]{BRZ1}
\begin{equation}
\label{mortruss}
\sigma_{B}\circ f=f\circ \sigma_{H}
\end{equation} 
holds.

Hopf trusses  together with morphisms of Hopf trusses form a category which we denote by {\sf HTr}. It is obvious that Hopf braces with morphisms of Hopf braces form a category which we denote by {\sf HBr} that is a subcategory of ${\sf HTr}$.  
}
\end{definition}

\begin{example}
\label{EXHT1}
Given a  Hopf monoid $D=(D,\eta_{D},\mu_{D},\varepsilon_{D},\delta_{D},\lambda_{D})$ and a comonoid endomorphism $q:D\rightarrow D$ satisfying 
\begin{equation}\label{condBidemp}
\mu_{D}\circ(q\otimes q)=q\circ\mu_{D}\circ (q\otimes D), 
\end{equation}
we have that  ${\mathbb D}_{q}=(D, D_{q}, \sigma_{D}^{q})$ is a Hopf truss where  $D_{q}$ is the non-unital bimonoid 
\begin{gather*}
D_{q}=(D,\mu_{D_q}=\mu_{D}\circ (q\otimes D),\varepsilon_{D},\delta_{D}),
\end{gather*}
and $\sigma_{D}^q=q$. Note that in this case  $\Gamma_{D}^q=\varepsilon_{D}\otimes D$.  
		
In this setting,  a family of morphisms $q\colon D\rightarrow D$ satisfying \eqref{condBidemp} is made up of idempotent Hopf monoid endomorphisms of $D$. So, we can construct examples of Hopf trusses   working with idempotent Hopf monoid endomorphisms of $D$. For example, if $D$ is Hopf monoid  that factorizes by the semidirect product of two Hopf monoids $A$, $H$ and $\omega_{D}:A\ltimes H\rightarrow  D$ is the corresponding isomorphism of Hopf monoids, then 
$$q=\omega_{D}\circ (\eta_{A}\otimes ((\varepsilon_{A}\otimes H)\circ \omega_{D}^{-1}):D\rightarrow D$$
is an idempotent morphism of Hopf monoids.  Remember that, in the particular case of groups (Hopf monoids in {\sf Set}), it  is well known that the set of  idempotent endomorphisms $q$ of a group $D$ are in one-to-one correspondence with the semidirect-product decompositions $A\ltimes H$ of $D$ where $A=Ker(f)$, $H= q(D)$ and $\varphi_{A}: H\times A\rightarrow A$ is the adjoint action of $H$ on $A$, i.e.,   $\varphi_A(h,a)=h a h^{-1}$.  The product on $A\ltimes H$ is defined by $(a,h)(b,l)=(a\varphi_{A}(h,b), hl)$.
		
\end{example}
	
\begin{example}
\label{EXHT2}
Let ${\mathbb H}=(H_{1}, H_{2},  \sigma_{H})$ be a Hopf truss. Assume that there exists a comonoid endomorphism $q:H\rightarrow H$ satisfying
\begin{equation}\label{condBidemp1}
\mu_{H}^2\circ(q\otimes q)=q\circ\mu_{H}^2\circ (q\otimes H).
\end{equation}

Then, ${\mathbb H}^{q}=(H_{1}, H_{2}^q, \sigma_{H}^q)$  is a Hopf truss where $\mu_{H_{2}^q}=\mu_{H}^2\circ (q\otimes H)$ and $\sigma_{H}^q=\sigma_{H}\circ q$.  Note that in this case $\Gamma_{H_{1}}^{\sigma_{H}^q}=\Gamma_{H_{1}}^{\sigma_{H}}\circ (q\otimes H)$.

For example, if ${\mathbb H}=(H_{1}, H_{2})$ is a Hopf brace and $q:H_{2}\rightarrow H_{2}$ is an idempotent morphism of Hopf monoids, then (\ref{condBidemp1}) holds and, as a consequence, ${\mathbb H}^{q}=(H_{1}, H_{2}^{q}, \sigma_{\sigma_{H}^q})$ is a Hopf truss where $H_{2}^{q}=(H,\mu_{H_{2}^q}, \varepsilon_{H}, \delta_{H} )$ is the associated non-unital bimonoid. 
\end{example}

\begin{example}
\label{EXHT3}
Following  \cite[Definition 3.5]{LW2}, given a  cocommutative Hopf monoid $$D=(D,\eta_{D},\mu_{D},\varepsilon_{D},\delta_{D},\lambda_{D})$$ and a Hopf monoid endomorphism $\phi\colon D\rightarrow D$, a $\phi$-twisted operator is a comonoid morphism $\Upsilon \colon D\rightarrow D$ such that the equation
\begin{equation*}
\mu_{D}\circ(\Upsilon\otimes\Upsilon)=\Upsilon\circ\mu_{D}\circ((\mu_{D}\circ(\Upsilon\otimes D))\otimes (\lambda_{D}\circ \phi\circ \Upsilon))\circ(D\otimes c_{D,D})\circ(\delta_{D}\otimes D
\end{equation*}
holds.  By \cite[Proposition 3.6]{LW2}, if $\Upsilon\colon D\rightarrow D$ is a $\phi$-twisted operator, then the triple $(D,\Upsilon,\phi\circ \Upsilon)$ is a Rota-Baxter system (see  \cite[Definition 3.1]{LW2})  and, as a consequence, applying \cite[Proposition 3.8]{LW2} it is obtained that $\mathbb{D}_{\Upsilon}=(D,D_{\Upsilon},\sigma_{D}^{\Upsilon})$ is a cocommutative Hopf truss where
\begin{gather*}
\sigma_{D}^{\Upsilon}:=\mu_{D}\circ (\Upsilon\otimes(\lambda_{D}\circ\phi\circ \Upsilon))\circ\delta_{D}=\Upsilon\ast(\lambda_{D}\circ\phi\circ \Upsilon)
\end{gather*}
and $H_{\Upsilon}$ is the non-unital bimonoid
\begin{gather*}
H_{\Upsilon}=(H,\mu_{\Upsilon},\varepsilon_{D},\delta_{D}),
\end{gather*}
being
\begin{gather*}
\mu_{\Upsilon}:=\mu_{D}\circ((\mu_{D}\circ(\Upsilon\otimes D))\otimes(\lambda_{D}\circ \phi\circ \Upsilon))\circ(D\otimes c_{D,D})\circ(\delta_{D}\otimes D).
\end{gather*}
\end{example}

\begin{remark}
\label{ex-T}
Following \cite{BRZ1} a skew truss is a set $T$ with two binary operations $\diamond_{1} $ and  $\diamond_{2}$ and a map $\omega_{T} :T\rightarrow T$ (called the cocycle) such that the pair $T_{1}=(T, \diamond_{1})$ is a group with unit $1_{\diamond_{1}}$, $T_{2}=(T, \diamond_{2})$ is a semigroup  and the following identity 
\begin{equation}
\label{dia-dia}
a\diamond_{2} (b\diamond_{1} c)=
(a\diamond_{2} b)\diamond_{1} \omega(a)^{\diamond_{1}}\diamond_{1} (a\diamond_{2} c)
\end{equation}
holds for all $a,b,c\in T$.  We will denote the previous skew truss by ${\mathbb T}=(T_{1}, T_{2}, \omega_{T})$.   A morphism $f$ between two skew trusses ${\mathbb T}=(T_{1}, T_{2},  \omega_{T})$ and ${\mathbb S}=(S_{1}, S_{2},  \omega_{S})$ is a map $f$ between the two underlying sets such that $f$ is a morphism of groups between $T_{1}$ and $S_{1}$ and of semigroups between $T_{2}$ and $S_{2}$. Then, by \cite[Proposition 2.8]{BRZ1}, the equality  $\omega_{S}\circ f=f\circ \omega_{T}$ holds. With {\sf SkTr} we will denote the category of skew trusses.  Note that, taking into account the previous lines  {\sf SkTr}={\sf HTr} in the category {\sf Set}.

Let ${\mathbb F}$ be a field. Let ${\mathbb T}$ be a skew truss. Then ${\mathbb F}[T]$ admits a structure of Hopf truss in ${\mathbb F}$-{\sf Vect} where the products, coproduct, counit and antipode are defined as in Remark \ref{t-exe} and the comonoid morphism $\sigma_{{\mathbb F}[T]}$ is the linear extension of $\omega_{T}$. Also, if $f$ is a morphism between skew trusses, then its linear extension is a morphism of Hopf trusses.  As a consequence, there exists a  functor $${\sf P}_{st}: {\sf SkTr} \rightarrow  {\sf HTr}$$ 
given by
$${\sf P}_{st}({\mathbb T})=({\sf L}_{g}(T_{1}), {\sf L}_{sg}(T_{2}),  \sigma_{{\mathbb F}[T]})$$
where ${\sf L}_{g}$, ${\sf L}_{sg}$ are the functors defined in Remark \ref{t-exe} and $\lambda_{{\mathbb F}[T]}=(\;)^{\diamond_{1}}$.

Let ${\mathbb F}$ be a field and ${\mathbb H}=(H_{1}, H_{2},  \sigma_{H})$ a Hopf truss in 
${\mathbb F}$-{\sf Vect}.  Let ${\sf G}(H)$ be the set of grouplike elements of the comonoid $(H,\varepsilon_{H} ,\delta_{H})$. 
In light of Remark \ref{t-exe}, on the one hand we have that ${\sf G}(H)$ is a group. On the other hand, ${\sf G}(H)$ is a semigroup. As $\sigma_{H}$ is a comonoid morphism, $\sigma_{H}(h)\in {\sf G}(H)\; \forall\; h\in {\sf G}(H)$. If we denote by $\omega_{\sigma_{H}}$ the restriction of $\sigma_{H}$ to ${\sf G}(H)$, then we have that $({\sf G}_{g}(H_{1}), {\sf G}_{sg}(H_{2}),  \omega_{\sigma_{H}})$ is an object in  {\sf SkTr} because equality \ref{dia-dia} holds as a consequence of (iii) of Definition \ref{H-truss}. By the functoriality of ${\sf G}_{g}$ and ${\sf G}_{sg}$ any morphism $f:{\mathbb H} \to {\mathbb B}  $ of Hopf trusses induces a morphism of skew trusses between $({\sf G}_{g}(H_1), {\sf G}_{sg}(H_2), \omega_{\sigma_{H}})$ and  $({\sf G}_{g}(B_1), {\sf G}_{sg}(B_2), \omega_{\sigma_{B}})$ defined by ${\sf G}_{g}(f)$ or by  ${\sf G}_{sg}(f)$. Therefore $${\sf R}_{ht}: {\sf HTr}\rightarrow {\sf SkTr}$$ defined by ${\sf R}_{ht}({\mathbb H})=({\sf G}_{g}(H_{1}), {\sf G}_{sg}(H_{2}),  \omega_{\sigma_{H}})$ on objects and by ${\sf R}_{ht}(f)={\sf G}_{sg}(f)$ on morphisms is a functor between {\sf HTr} and {\sf SkTr}.
	
\end{remark}

\begin{definition}
Let ${\mathbb F}$ be a field and let  ${\mathbb H}$ be a Hopf truss in ${\sf Vect}_{\mathbb F}$. We will say that ${\mathbb H}$ is pointed cosemisimple if the its subjacent coalgebra $(H,\varepsilon_{H}, \delta_{H})$ is pointed and cosemisimple.
\end{definition}

\begin{theorem}
\label{111}
Let ${\sf P}_{st}$ and ${\sf R}_{ht}$ be the functors defined in the previous remark. Then, ${\sf P}_{st}\dashv {\sf R}_{ht}$ and this adjunction induces an equivalence of categories between  {\sf SkTr} and the full subcategory of {\sf HTr} of all pointed cosemisimple   Hopf trusses.
\end{theorem}

\begin{proof} 
Let ${\mathbb T}$ be an object in {\sf SkTr} and let ${\mathbb H}$ be an object in {\sf HTr}. We define a map 
$$\Gamma_{{\mathbb T},{\mathbb H}}:{\sf Hom}_{\sf HTr}({\sf P}_{st}({\mathbb T}), {\mathbb H})\rightarrow {\sf Hom}_{\sf SkTr}({\mathbb T}, {\sf R}_{ht}({\mathbb H}))$$
in the following way: if $q:{\sf P}_{st}({\mathbb T})\rightarrow {\mathbb H}$ is a morphism in ${\sf Hom}_{\sf HTr}({\sf P}_{st}({\mathbb T}), {\mathbb H})$, then $\Gamma_{{\mathbb T},{\mathbb H}}(q):T\rightarrow {\sf G(H)}$ is the map defined by $\Gamma_{{\mathbb T},{\mathbb H}}(q)(a)=q(a)$. Using that $q$ is a Hopf monoid  morphism between ${\sf L}_{g}(T_{1})$ and $H_{1}$ and a non-unital bimonoid morphism between  ${\sf L}_{sg}(T_{2})$ and $H_{2}$, we obtain that $\Gamma_{{\mathbb T},{\mathbb H}}(q)$ is a group morphism between $T_{1}$ and ${\sf G}_{g}(H_{1})$ and a semigroup morphism between $T_{2}$ and ${\sf G}_{sg}(H_{2})$. Therefore, $\Gamma_{{\mathbb T},{\mathbb H}}(q)$ is well defined. It is easy to show that it is natural in both components and injective. Moreover, it is surjective because, if $h:{\mathbb T}\rightarrow {\sf R}_{ht}({\mathbb H})$ is a morphism in {\sf SkTr}, then $h=\Gamma_{{\mathbb T},{\mathbb H}}(l_{h})$,  where $l_{h}$ is the unique morphisms in {\sf HTr} such that $l_{h}(a)=h(a)$ for all $a\in T$. Indeed, it is easy to show that $l_{h}$  is a Hopf monoid morphism between ${\sf L}_{g}(T_{1})$ and $H_{1}$ and a morphism of non-unital bimonoids between ${\sf L}_{sg}(T_{2})$ and $H_{2}$. 
Finally, the identity $h=\Gamma_{{\mathbb T},{\mathbb H}}(l_{h})$ follows trivially from the definition of $l_{h}$.

Therefore  ${\sf P}_{st}\dashv {\sf R}_{ht}$ and the unit of this adjuction  is the identity while the counit is not a natural isomorphism in general.  In any case, if ${\mathbb H}$ is a pointed cosemisimple  Hopf truss, then  the counit is an isomorphism, as the counits for the adjunctions given in Remark \ref{t-exe} are, and thus we have the equivalence.

\end{proof}

\section{Hopf trusses and generalized invertible 1-cocycles}

In \cite{AGV} the authors proved that there exists a closed relation between Hopf braces and 1-cocycles. In  this section we will prove that this connection remains valid for Hopf trusses. First we will introduce the notion of generalized invertible 1-cocycle between a non-unital bimonoid $B$ and a Hopf monoid $H$ in the braided monoidal category ${\sf C}$.

\begin{definition}
\label{1-cocy} 
Let $H=(H,\eta_{H}, \mu_{H},\varepsilon_{H}, \delta_{H}, \lambda_{H})$ be a Hopf monoid in {\sf C} and let  $B=(B, \mu_{B},\varepsilon_{B}, \delta_{B})$ be a non-unital bimonoid in {\sf C}. Assume that $H$ is a non-unital left $B$-module monoid with action $\phi_{H}$. Let $\pi:B\rightarrow H$ be comonoid morphism. We will say that $\pi$ is an generalized invertible 1-cocycle if it is an isomorphism and there exist a comonoid endomorphism $\theta_{\pi}: B\rightarrow B$ such that 
\begin{equation}
\label{1-c}
\pi\circ \mu_{B}=\mu_{H}\circ ((\pi\circ \theta_{\pi})\otimes \phi_{H})\circ (\delta_{B}\otimes \pi)
\end{equation}
holds.	

Let $\pi:B\rightarrow H$ and $\pi^{\prime}:B^{\prime}\rightarrow H^{\prime}$ be generalized invertible 1-cocycles. A morphism between them is a pair $(f,g)$ where $f:B\rightarrow B^{\prime}$ is a morphism of non-unital bimonoids  and $g:H\rightarrow H^{\prime}$ is a morphism of Hopf monoids satisfying the following identities:
\begin{equation}
\label{1-c1}
f\circ \theta_{\pi}=\theta_{\pi^{\prime}}\circ f,
\end{equation}
\begin{equation}
\label{1-c2}
g\circ \pi=\pi^{\prime}\circ f,
\end{equation}
\begin{equation}
\label{1-c3}
g\circ \phi_{H}=\phi_{H^{\prime}}\circ (f\otimes g).
\end{equation}

Then, with these morphisms, generalized invertible 1-cocycles form a category denoted by {\sf GIC}. In the following lines an object in {\sf GIC} will also be denoted by the triple $(\pi: B\rightarrow H, \theta_{\pi})$.

Note that if $(\pi: B\rightarrow H, \theta_{\pi})$ is a generalized invertible 1-cocycle such that $B$ is a Hopf monoid, $(H, \phi_{H})$ is a left $B$-module monoid and $\theta_{\pi}=id_{B}$, then $(\pi: B\rightarrow H, \theta_{\pi})$ is an invertible 1-cocycle in the sense of \cite{AGV}. If we denote the category of invertible 1-cocycles by {\sf IC}, then  it is obvious that it is a subcategory of {\sf GIC}.

\end{definition}

\begin{theorem}
\label{EGIHT}
The categories {\sf GIC} and {\sf HTr} are equivalent.
\end{theorem}

\begin{proof}
Let ${\mathbb H}=(H_{1}, H_{2},  \sigma_{H})$ be an object in {\sf HTr}. Then, $(id_{H}:H_{2}\rightarrow H_{1}, \theta_{id_{H}}=\sigma_{H})$ is a generalized  invertible 1-cocycle. Indeed, 
trivially $id_{H}$ is a comonoid isomorphism and $H_{1}$ is a non unital left $H_{2}$-module monoid  with the action $\Gamma_{H_{1}}^{\sigma_{H}}$ defined in (iii) of Definiton (\ref{H-truss}). Finally,  
\begin{itemize}
\item[ ]$\hspace{0.38cm}\mu_H^1\circ ((id_{H}\circ \theta_{id_{H}})\otimes \Gamma_{H_{1}}^{\sigma_{H}})\circ (\delta_{H} \otimes id_{H}) $
\item [ ]$=\mu_H^1\circ (\sigma_{H}\otimes (\mu_H^1\circ ((\lambda_H\circ \sigma_H)\otimes \mu_{H}^2)\circ (\delta_{H}\otimes H)))\circ (\delta_H \otimes H)   $ {\scriptsize (by the definition of $\Gamma_{H_{1}}^{\sigma_{H}}$)}
\item [ ]$=\mu_H^1\circ ((\mu_H^1\circ (\sigma_H\otimes (\lambda_H\circ \sigma_H)\circ \delta_H))\ot \mu_H^2)\circ \delta_H$ {\scriptsize (by the associativity of $\mu_H^1$ and the coassociativity of}
\item[ ]$\hspace{0.38cm}$ {\scriptsize  $\delta_H$)}
\item [ ]$=\mu_H^1\circ ((id_{H}\ast \lambda_H)\ot \mu_H^2)\circ \delta_H$ {\scriptsize  (by the condition of comonoid morphism for $\sigma_H$)}
\item [ ]$= id_{H}\circ \mu_H^2  $ {\scriptsize  (by (\ref{antipode}) and the properties of  $\varepsilon_H$ and $\eta_H$).}
\end{itemize}

On the other hand, let ${\mathbb H}=(H_1, H_{2},  \sigma_{H})$ and ${\mathbb H}^{\prime}=(H_1^{\prime}, H_{2}^{\prime},  \sigma_{H^{\prime}})$ be objects in {\sf HTr} and let $f:{\mathbb H}\rightarrow {\mathbb H}^{\prime}$ be a morphism between them. The pair $(f,f)$ is a morphism in {\sf GIC} between $(id_{H}:H_{2}\rightarrow H_1, \sigma_{H})$ and $(id_{H^{\prime}}:H_{2}^{\prime}\rightarrow H_1^{\prime}, \sigma_{H^{\prime}})$ because $f:H_1\to H_1^{\prime}$ is a Hopf  monoid morphism, $f:H_2\to H_2^{\prime}$ is a non-unital bimonoid morphisn,  (\ref{1-c1}) holds trivially, (\ref{1-c2}) holds by (\ref{mortruss}),  and, finally, (\ref{1-c3}) follows from 
\begin{itemize}
\item[ ]$\hspace{0.38cm}f\circ \Gamma_{H_{1}}^{\sigma_{H}} $
\item [ ]$=\mu_H^1\circ ((\lambda_H\circ f\circ\sigma_H)\otimes (\mu_{H^{\prime}}^2\circ (f\otimes f)))\circ (\delta_{H}\otimes H)$ {\scriptsize (by the condition of multiplicative morphism for $f$}
\item[ ]$\hspace{0.38cm}${\scriptsize  and (\ref{morant}))}
\item [ ]$=\mu_H^1\circ ((\lambda_H\circ \sigma_{H^{\prime}}\circ f)\otimes (\mu_{H^{\prime}}^2\circ (f\otimes f)))\circ (\delta_H \otimes H)   $ {\scriptsize (by (\ref{mortruss}))} 
\item [ ]$=\Gamma_{H_{1}^{\prime}}^{\sigma_{H^{\prime}}}\circ (f\otimes f)$ {\scriptsize  (by  the condition of comonoid morphism for $f$)}.
\end{itemize}

Therefore, there exists a functor $${\sf E}:{\sf HTr}\rightarrow  {\sf GIC}$$ defined on objects by ${\sf E}({\mathbb H})=(id_{H}:H_{2}\rightarrow H_1, \sigma_{H})$ and on morphisms by ${\sf E}(f)=(f,f).$

The next step is to define the functor $${\sf Q}:\;{\sf GIC} \rightarrow  {\sf HTr}.$$  

Let  $(\pi: B\rightarrow H, \theta_{\pi})$ be an object in {\sf GIC}. Define $\mu_{H}^{\pi}:=\pi\circ\mu_{B}\circ (\pi^{-1}\otimes \pi^{-1})$ and $\sigma_{\pi}:=\pi\circ \theta_{\pi}\circ \pi^{-1}$. Then,  ${\mathbb H}_{\pi}=(H, H_{\pi},  \sigma_{\pi})$, where $H_{\pi}=(H, \mu_{H}^{\pi}, \varepsilon_{H}, \delta_{H})$, is an object in {\sf HTr}. Indeed, trivially $H_{\pi}$ is a non-unital bimonoid in {\sf C} because $\pi$ is a comonoid isomorphism and $B$ is a non-unital bimonoid. The morphism $\sigma_{\pi}$ is a comonoid morphism because it is defined as a composition of comonoid morphisms. Finally, 

\begin{itemize}
\item[ ]$\hspace{0.38cm}\mu_{H}\co (\mu_{H}^{\pi}\ot \Gamma_{H}^{\sigma_{\pi}} )\co (H\ot c_{H,H}\ot H)\co (\delta_{H}\ot H\ot H) $
\item [ ]$=\mu_{H}\circ ((\mu_{H}\circ ((\pi\circ \theta_{\pi})\otimes \phi_{H})\circ (\delta_{B}\otimes H))\otimes (\mu_{H}\circ ((((\lambda_{H}\ast id_{H}) \circ \pi\circ  \theta_{\pi})\otimes \phi_{H})  \circ (\delta_{B}\otimes H)))$
\item[ ]$\hspace{0.38cm}\circ (B\otimes c_{B,H}\otimes H) \circ ((\delta_{B}\circ \pi^{-1})\otimes H\otimes H)$ {\scriptsize (by (\ref{1-c}), associativity of $\mu_{H}$, coassociativity of $\delta_{H}$ and naturality}
\item[ ]$\hspace{0.38cm}$ {\scriptsize  of $c$)}
\item [ ]$= \mu_{H}\circ ((\pi\circ \theta_{\pi}) \otimes  (\mu_{H}\circ (\phi_{H}\otimes \phi_{H})\circ (B\otimes c_{B,H}\otimes H)\circ (\delta_{B}\otimes H\otimes H)))\circ  ((\delta_{B}\circ \pi^{-1})\otimes H\otimes H)$  
\item[ ]$\hspace{0.38cm}${\scriptsize (by (\ref{antipode}), the condition of comonoid morphism for $\theta_{\pi}$ and $\pi$, the properties of the unit $\eta_{H}$ and the counit $\varepsilon_{B}$ and}
\item[ ]$\hspace{0.38cm}${\scriptsize the coassociativity of $\delta_{H}$)}
\item [ ]$= \mu_{H}^{\pi}\circ (H\otimes \mu_{H})  $ {\scriptsize  (by the condition of non-unital left $B$-module monoid for $H$ and (\ref{1-c}))}.
\end{itemize}

Also, if $(f,g)$ is a morphism in {\sf GIC} between $(\pi: B\rightarrow H, \theta_{\pi})$ and $(\pi^{\prime}: B^{\prime}\rightarrow H^{\prime}, \theta_{\pi^{\prime}})$, then $g$ is a morphism in {\sf HTr} between ${\mathbb H}_{\pi}$ and ${\mathbb H}_{\pi^{\prime}}^{\prime}$ because, $g:H\to H^{\prime}$ is a morphism of Hopf monoids and $g:H_{\pi}\to H_{\pi^{\prime}}^{\prime}$ is a morphism in {\sf bimod} because
\begin{itemize}
\item[ ]$\hspace{0.38cm}g\circ \mu_{H}^{\pi} $
\item [ ]$=g\circ \pi\circ \mu_{B}\circ (\pi^{-1}\otimes \pi^{-1}) $ {\scriptsize (by definition of $\mu_{H}^{\pi}$)}
\item [ ]$= \pi^{\prime}\circ f\circ \mu_{B}\circ (\pi^{-1}\otimes \pi^{-1})  $ {\scriptsize (by (\ref{1-c2}))}
\item [ ]$= \pi^{\prime}\circ \mu_{B^{\prime}}\circ ((f\circ\pi^{-1})\otimes (f\circ\pi^{-1}))$ {\scriptsize  (by the condition of  multiplicative morphism for $f$)}
\item [ ]$= \pi^{\prime}\circ \mu_{B^{\prime}}\circ ((\pi^{\prime -1}\circ g)\otimes (\pi^{\prime -1}\circ g)) $ {\scriptsize (by (\ref{1-c2}))}
\item [ ]$=\mu_{H^{\prime}}^{\pi^{\prime}}\circ (g\otimes g)$ {\scriptsize (by definition of $\mu_{H^{\prime}}^{\pi^{\prime}}$).}
\end{itemize}
 
As a consequence of these facts, we have a functor ${\sf Q}:\;{\sf IC}\rightarrow  {\sf HTr}$ defined by ${\sf Q}((\pi: B\rightarrow H, \theta_{\pi}))={\mathbb H}_{\pi}$  on objects and by 
${\sf Q}((f,g))=g$ on morphisms. 

These functors induce an equivalence between the two categories because, clearly, ${\sf QE}=id_{{\sf HTr}}$. On the other hand,  we have that  
$$(\pi, id_{H}): (\pi: B\rightarrow H, \theta_{\pi}) \rightarrow (id_{H}:H_{\pi}\rightarrow H, \sigma_{\pi})={\sf EQ}((\pi: B\rightarrow H, \theta_{\pi}))$$ is an isomorphism in {\sf GIC}  because $\pi:B\rightarrow H_{\pi}$ is a isomorphism between $B$ and $H_{\pi}$ in {\sf bimod}, $\pi\circ \theta_{\pi}=\sigma_{\pi}\circ \pi$, (\ref{1-c2}) holds trivially, and

\begin{itemize}
	\item[ ]$\hspace{0.38cm}\Gamma_{H}^{\sigma_{\pi}}\circ (\pi \otimes id_{H}) $
	\item [ ]$= \mu_{H}\circ ((\lambda_{H}\circ \pi\circ \theta_{\pi}\circ \pi^{-1})\otimes (\pi\circ \mu_{B}\circ (\pi^{-1}\otimes \pi^{-1})))\circ ((\delta_{H}\circ \pi)\otimes H)$ {\scriptsize (by definition of $\sigma_{\pi}$ and $\mu_H^{\pi}$)}
	\item [ ]$= \mu_H\circ (((\lambda_H\ast id_{H})\circ \pi\circ \theta_{\pi})\otimes \phi_{H})\circ (\delta_{B}\otimes H)$ {\scriptsize (by (\ref{1-c}), the condition of comonoid morphism for $\theta_{\pi}$ and}
	\item[ ]$\hspace{0.38cm}$ {\scriptsize  $\pi$)}
	\item [ ]$=\phi_{H} $ {\scriptsize (by (\ref{antipode}), the condition of comonoid morphism for $\theta_{\pi}$ and $\pi$ and the properties of the unit $\eta_{H}$ and the counit}
	\item[ ]$\hspace{0.38cm}$ {\scriptsize  $\varepsilon_{B}$)}.
\end{itemize}

Therefore, ${\sf EQ}\backsimeq {\sf Id}_{{\sf GIC}}$. 

\end{proof}

\begin{example}
For example, thanks to Theorem \ref{EGIHT} we can assert the following: In the conditions  of Example \ref{EXHT1},  $(id_{D}:D_{q}\rightarrow D, q)$ is a generalized invertible 1-cocycle. Also,  using  Example \ref{EXHT2}, we obtain that  $(id_{H}:H_{2}^q\rightarrow H_1, \sigma_{H}\circ q)$ is a generalized invertible 1-cocycle. Finally,  Example  
\ref{EXHT3} implies that  $(id_{D}:D_{\Upsilon}\rightarrow D, \Upsilon\ast (\lambda_{D}\circ \phi\circ \Upsilon))$ is a generalized invertible 1-cocycle for a cocommutative Hopf monoid $D$.
\end{example}

\begin{example}\label{EXHT1-I1C}
	Let $H$ be a Hopf monoid and $B$ a non unital bimonoid. Suppose that $\pi:B\to H$ is an isomorphism of comonoids and $\theta_{\pi}:B\to B$ is a morphism of comonoids satisfying 
	\begin{equation}\label{EXHT1-I1C-cond1}
		\pi\circ\mu_B = \mu_H\circ ((\pi\circ\theta_{\pi})\ot \pi)
	\end{equation}
	and 
	\begin{equation}\label{EXHT1-I1C-cond2}
		\mu_B\circ (B\ot \theta_{\pi}) = \theta_{\pi}\circ \mu_B.
	\end{equation}
	
These data define a generalized invertible 1-cocycle considering the trivial action of $B$ on $H$ $\varphi_H= \varepsilon_B\ot H$. Moreover, if we apply the functor ${\sf Q}$ obtained in the proof of Theorem \ref{EGIHT}, then the corresponding Hopf truss is of the type defined in Example \ref{EXHT1} for $q= \pi\circ \theta_{\pi}\circ \pi^{-1}$, i.e., ${\mathbb H}_{\pi}={\mathbb H}_{q}.$
\end{example}

\begin{example}\label{EXHT2-I1C}
Suppose that $B$ and $H$ are Hopf monoids and let $\pi:B\to H$ be a generalized invertible 1-cocycle. If $p:B\to B$ is a comonoid morphism satisfying equality (\ref{condBidemp}) we obtain a non unital bimonoid $B_{p} = (B, \mu_{B_{p}}=  \mu_B\circ (p\ot B), \varepsilon_B,\delta_B, )$, as mentioned in Example \ref{EXHT1}. We also obtain a new generalized invertible 1-cocycle $\pi_{p}:B_{p}\to H$ for the non unital $B_{p}$-module structure of $H$ given by $\phi_H^{p}= \phi_H\circ (p\ot H):B\ot H\to H$. The isomorphism of comonoids is given by $\pi_{p} = \pi$ and the corresponding comonoid morphism is $\theta_{\pi_{p}} = \theta_{\pi}\circ p:B_{p}\to B_{p}$ . 
	
The Hopf truss obtained following Theorem \ref{EGIHT} is of the type described in Example \ref{EXHT2} for $q = \pi\circ p\circ \pi^{-1}$, i.e., ${\mathbb H}_{\pi_{p}}={\mathbb H}_{\pi}^{q}.$
\end{example}

\begin{remark}
As is easy to see, the previous theorem is a generalization of  \cite[Theorem 1.12]{AGV} because if we particularize this equivalence to the Hopf brace setting, then we obtain the quoted result.
\end{remark}

\section{Modules for Hopf trusses and generalized invertible 1-cocycles}

In this section we introduce the notion of module over a Hopf truss and over a generalized invertible 1-cocycle. Also we will prove that, thanks to the categorical equivalence of Theorem \ref{EGIHT}, there exist some interesting categorical connections between these categories of modules.

\begin{definition}
\label{l-mod}
Let ${\Bbb H}$ be a Hopf truss. A left ${\Bbb H}$-module is a triple $(M,\psi_{M}^{1}, \psi_{M}^{2})$, where $(M,\psi_{M}^{1})$ is a left $H_{1}$-module, $(M, \psi_{M}^{2})$  is a non-unital left $H_{2}$-module and the following identity 
\begin{equation}
\label{mod-l1}
\psi_{M}^{2}\co (H\ot \psi_{M}^{1})=\psi_{M}^1\co (\mu_{H}^{2}\ot \Gamma_{M}^{\sigma_{H}})\co (H\ot c_{H,H}\ot M)\co (\delta_{H}\ot H\ot M)
\end{equation}
holds, where 
$$\Gamma_{M}^{\sigma_{H}}=\psi_{M}^{1}\co ((\lambda_{H}\circ \sigma_{H})\ot \psi_{M}^{2})\co (\delta_{H}\ot M).$$
		
Given two left ${\Bbb H}$-modules  $(M,\psi_{M}^1, \psi_{M}^{2})$  and  $(N,\psi_{N}^1, \psi_{N}^{2})$, a morphism $f:M\rightarrow N$  is called a morphism of left ${\Bbb H}$-modules if  $f$ is a morphism of left $H_{1}$-modules and left non-unital $H_{2}$-modules. Left ${\Bbb H}$-modules  with morphisms of left ${\Bbb H}$-modules  form a category which we denote by $\;_{\Bbb H}${\sf Mod}. 
\end{definition}

\begin{example}
\label{ex-mtr}
Let ${\Bbb H}$ be a Hopf truss. The triple $(H, \psi_{H}^1=\mu_{H}^{1}, \psi_{H}^2=\mu_{H}^{2})$ is an example of  left ${\Bbb H}$-module. Also, if $K$ is the unit object of ${\sf C}$, then $(K, \psi_{K}^1=\varepsilon_{H}, \psi_{K}^2=\varepsilon_{H})$ is a left ${\Bbb H}$-module called the trivial module. 
	
If $X$ is an object in {\sf C}, then the triple 
$${\mathbb H}\otimes X=(H\otimes X, \psi_{H\otimes X}^1=\mu_{H}^{1}\otimes X, \psi_{H\otimes X}^2=\mu_{H}^{2}\otimes X)$$ is an example of left ${\Bbb H}$-module. Also, if $f:X\rightarrow X^{\prime}$ is a morphism in {\sf C}, then $H\otimes f$ is a morphism in $\;_{\Bbb H}${\sf Mod} between ${\mathbb H}\otimes X$ and  ${\mathbb H}\otimes X^{\prime}$. Therefore, there exist a functor, called the induction functor, 
$${\mathbb H}\otimes -: {\sf C}\rightarrow \;_{\Bbb H}{\sf Mod}$$
defined on objects by ${\mathbb H}\otimes -(X)={\mathbb H}\otimes X$ and by 
${\mathbb H}\otimes -(f)={\mathbb H}\otimes f$ on morphisms.
\end{example}

\begin{example}\label{EXHT2-Mod}
Let $D$ be a Hopf monoid and $q:D\to D$ a comonoid morphism satisfying equality (\ref{condBidemp}), and let ${\mathbb D}_q$ be the Hopf truss defined in Example \ref{EXHT1}. If $(M, \varphi_M)$ is a left $D$-module, then $(M, \varphi_M, \varphi_M^{q} = \varphi_M\circ (q\ot M))$ is a left ${\mathbb D}_q$-module with $\Gamma_M^{q} = \varepsilon_D\ot M$. In particular, if $\pi:D\to H$ is an invertible 1-cocycle, then $H$ becomes a left ${\mathbb D}_q$-module. 
\end{example}

\begin{remark}
\label{mis-mod} If the a Hopf truss ${\mathbb H}$ is a Hopf brace and in Definition \ref{l-mod} we assume that a $(M, \psi_{M}^{2})$ is a left $H_2$-module, then we obtain the definition of module over a Hopf brace introduced in \cite{RGON}.

\end{remark}

\begin{remark}
\label{yo} 
Using the naturality of $c$ and the coassociativity of $\delta_{H}$, it is easy to show that  (\ref{mod-l1}) is equivalent to 
\begin{equation}
\label{mod-l1p}
\psi_{M}^{2}\co (H\ot \psi_{M}^{1})=\psi_{M}^1\co (\Lambda_{H_1}^{\sigma_{H}}\ot \psi_{M}^{2})\co (H\ot c_{H,H}\ot M)\co (\delta_{H}\ot H\ot M)
\end{equation}
where 
$$\Lambda_{H_1}^{\sigma_{H}}=\mu_{H}^{1}\circ (\mu_{H}^2\otimes (\lambda_{H}\circ \sigma_{H}))\circ (H\otimes c_{H,H})\circ (\delta_{H}\otimes H).$$

Also, the following equality
\begin{equation}
\label{GMH1}
\Gamma_{M}^{\sigma_{H}}\circ (H\otimes \psi_{M}^1)=\psi_{M}^1\circ (\Gamma_{H_{1}}^{\sigma_{H}}\otimes \Gamma_{M}^{\sigma_{H}})\circ (H\otimes c_{H,H}\otimes M)\circ (\delta_{H}\otimes H\otimes M).
\end{equation}
 holds. Indeed: 
\begin{itemize}
\item[ ]$\hspace{0.38cm} \psi_{M}^1\circ (\Gamma_{H_{1}}^{\sigma_{H}}\otimes \Gamma_{M}^{\sigma_{H}})\circ (H\otimes c_{H,H}\otimes M)\circ (\delta_{H}\otimes H\otimes M) $
\item [ ]$=\psi_{M}^1\circ ((\lambda_{H}\circ \sigma_{H})\otimes (\psi_{M}^1\co (\Lambda_{H_1}^{\sigma_{H}}\ot \psi_{M}^{2})\co (H\ot c_{H,H}\ot M)\co (\delta_{H}\ot H\ot M)))\circ (\delta_{H}\otimes H\otimes M) $ 
\item[ ]$\hspace{0.38cm}$ {\scriptsize (by naturality of $c$, coassociativity of $\delta_{H}$, the condition of left $H_{1}$-module for $(M,\psi_{M}^1)$ and the associativity of}
\item[ ]$\hspace{0.38cm}$ {\scriptsize  $\mu_{H}^{1}$)}
\item [ ]$=\Gamma_{M}^{\sigma_{H}}\circ (H\otimes \psi_{M}^1) $ {\scriptsize (by (\ref{mod-l1p}))}.
\end{itemize}	
\end{remark}

\begin{definition}
\label{def-pi-module}
Let $(\pi:B\to H, \theta_{\pi})$ be a generalized invertible 1-cocycle. A left module over $(\pi:B\to H, \theta_{\pi})$ is a 6-tuple \((M,N, \phi_M, \varphi_M,  \phi_N, \gamma)\) where
\begin{enumerate}
\item[(i)] $\phi_M:B\otimes M\rightarrow M$ is a morphism in {\sf C}.
\item[(ii)] $(M, \varphi_M)$ is a  left $H$-module.
\item[(iii)] $(N, \phi_N)$ is a non-unitary left $B$-module.
\item[(iv)] The equality 
\begin{equation}
\label{p-v}
\phi_M\circ (B\ot \varphi_M) = \varphi_M\circ (\phi_H\ot \phi_M)\circ (B\ot c_{A,H}\ot M)\circ (\delta_B\ot H\ot M).
\end{equation}
holds.
\item[(v)] $\gamma:N\to M$ is an isomorphism in $\sf C$ such that 
\begin{equation}\label{eq-gamma}
\gamma\circ \phi_N = \varphi_M\circ ((\pi \circ \theta_{\pi})\ot \phi_M)\circ (\delta_B\ot \gamma).
\end{equation}
\end{enumerate} 
\end{definition}

\begin{definition}
\label{def-pi-mor}	
Let \((M,N, \phi_M, \varphi_M,  \phi_N, \gamma)\) and 
\((M^{\prime},N^{\prime}, \phi_{M^{\prime}}, \varphi_{M^{\prime}},  \phi_{N^{\prime}}, \gamma^{\prime})\)
be left modules over a generalized invertible 1-cocycle $(\pi:B\to H, \theta_{\pi})$. A morphism between them is a pair $(h,l)$ of morphisms in {\sf C} such that:
\begin{itemize}
\item[(i)] The morphism $h:M\rightarrow M^{\prime}$ satisfies $f\circ \phi_M=\phi_{M^{\prime}}\circ (B\otimes f)$ and is left $H$-linear.
\item[(ii)] The morphism $l:N\rightarrow N^{\prime}$ is  left $B$-linear.
\item[(iii)] The following identity holds:
\begin{equation}\label{fg-g}
h\circ \gamma = \gamma^{\prime}\circ l.
\end{equation}
\end{itemize}

Note that, by (\ref{fg-g}), the morphism $l$ is determined by $h$ because $l=(\gamma^{\prime})^{-1}\circ h\circ \gamma.$
	
With the obvious composition of morphisms, left modules over a generalized  invertible 1-cocycle $(\pi:B\to H, \theta_{\pi})$ with action $\phi_{H}$ form a category that we will denote by $\;_{(\pi, \phi_{H}, \theta_{\pi})}${\sf Mod}.
\end{definition}

\begin{remark}
If  \((M,N, \phi_M, \varphi_M,  \phi_N, \gamma)\) is a left module over the generalized invertible 1-cocycle  $(\pi: B\rightarrow H, \theta_{\pi})$, by (\ref{eq-gamma}),  we obtain that $\phi_{N}$ is determined by $\phi_{M}$ and $\varphi_{M}$ in the following way:
\begin{equation}
\label{req-g1}
\phi_N =\gamma^{-1}\circ  \varphi_M\circ ((\pi \circ \theta_{\pi})\ot \phi_M)\circ (\delta_B\ot \gamma).
\end{equation}
	
Also, composing in both sides of the equality  (\ref{eq-gamma}) with $(((\lambda_H\circ \pi)\ot B)\circ \delta_B)\ot \gamma^{-1}$ on the right and with $\varphi_M$ on the left we obtain the identity
\begin{equation}
\label{req-g2}
\phi_M =\varphi_{M}\circ ((\lambda_{H}\circ \pi\circ \theta_{\pi})\otimes (\gamma\circ \phi_{N}))\circ (\delta_{B}\otimes \gamma^{-1})
\end{equation}
because:
\begin{itemize}
\item[ ]$\hspace{0.38cm}\varphi_{M}\circ ((\lambda_{H}\circ \pi\circ \theta_{\pi})\otimes (\gamma\circ \phi_{N}) ) \otimes (\delta_{B}\otimes \gamma^{-1})$
\item [ ]$= \varphi_{M}\circ ((\lambda_{H}\circ \pi\circ \theta_{\pi})\otimes (\varphi_M\circ ((\pi \circ \theta_{\pi})\ot \phi_M)\circ (\delta_B\ot \gamma)) \otimes (\delta_{B}\otimes \gamma^{-1}) $ {\scriptsize (by (\ref{eq-gamma}))}
\item [ ]$= \varphi_{M}\circ(((\lambda_{H}\ast id_{H})\circ \pi\circ \theta_{\pi}) \otimes \phi_{M})\otimes (\delta_{B}\otimes M)$ {\scriptsize (by  the condition of comonoid morphism for $\theta_{\pi}$ and $\pi$, the coassociativity of  $\delta_{B}$)}
\item [ ]$= \phi_{M}$ {\scriptsize (by (\ref{antipode}) and the unit counit properties))}.
\end{itemize}
\end{remark}

\begin{example}
\label{exmod1}
It is easy to see that, if $(\pi: B\rightarrow H, \theta_{\pi})$ is a generalized invertible 1-cocycle, then the 6-tuple \((H,B, \phi_H, \mu_H,  \mu_B, \pi)\) is an example of left module over $(\pi: B\rightarrow H, \theta_{\pi})$.
\end{example}

\begin{theorem}
\label{modfg}
Assume that  $(f,g)$ is a morphism  between the generalized invertible 1-cocycles $(\pi: B\rightarrow H, \theta_{\pi})$ and $(\pi: B^{\prime}\rightarrow H^{\prime}, \theta_{\pi^{\prime}})$. Then, there exists a functor 
$${\sf M}_{(f,g)}:\;_{(\pi^{\prime}, \phi_{H^{\prime}}, \theta_{\pi^{\prime}})}{\sf Mod}\;\rightarrow \;_{(\pi, \phi_{H}, \theta_{\pi})}{\sf Mod}$$
defined on objects by 
$${\sf M}_{(f,g)}((P,Q, \phi_P, \varphi_P,  \phi_Q, \tau))=(P,Q, \phi_P^{\pi}=\phi_P\circ (f\otimes P), \varphi_P^{\pi}=\varphi_P\circ (g\otimes P),  \phi_Q^{\pi}=\phi_Q\circ (f\otimes Q), \tau)$$
and on morphisms by the identity.
\end{theorem}

\begin{proof} The existence of the functor ${\sf M}_{(f,g)}$ is a consequence of the following facts: Trivially $(P, \varphi_P^{\pi})$ is a left $H$-module and  $(Q, \phi_Q^{\pi})$ is a non-unitary left $B$-modules and . Also,  
\begin{itemize}
\item[ ]$\hspace{0.38cm}\varphi_P^{\pi}\circ (\phi_H\ot \phi_P^{\pi})\circ (B\ot c_{B,H}\ot P)\circ (\delta_B\ot H\ot P) $
\item [ ]$=\varphi_P\circ ((g\circ \phi_{H})\otimes (\phi_{P}\circ (f\otimes P)))\circ (B\otimes c_{B,H}\otimes P)\circ (\delta_{B}\otimes H\otimes P )$ {\scriptsize (by definition of $\phi_P^{\pi}$ and $\varphi_P^{\pi}$)}
\item [ ]$=\varphi_P\circ ( \phi_{H^{\prime}}\otimes \phi_{P})\circ (B^{\prime}\otimes c_{B^{\prime},H^{\prime}}\otimes P)\circ (((f\otimes f)\circ \delta_{B})\otimes g\otimes P ) $ {\scriptsize (by (\ref{1-c3}) and naturality of $c$)}
\item [ ]$=\varphi_P\circ ( \phi_{H^{\prime}}\otimes \phi_{P})\circ (B^{\prime}\otimes c_{B^{\prime},H^{\prime}}\otimes P)\circ ((\delta_{B^{\prime}}\circ f)\otimes g\otimes P )$ {\scriptsize(by the comonoid morphism condition for}
\item[ ]$\hspace{0.38cm}$ {\scriptsize $f$)}
\item [ ]$=\phi_P^{\pi}\circ (B\ot \varphi_P^{\pi}) ${\scriptsize (by (\ref{p-v}))},
\end{itemize}
and
\begin{itemize}
\item[ ]$\hspace{0.38cm}\varphi_P^{\pi}\circ ((\pi\circ \theta_{\pi})\ot \phi_P^{\pi})\circ (\delta_B\ot \tau) $
\item [ ]$= \varphi_P\circ ((g\circ \pi\circ \theta_{\pi})\otimes (\phi_{P}\circ (f\otimes P)))\circ (\delta_{B}\otimes \tau)$ {\scriptsize (by definition of $\phi_P^{\pi}$ and $\varphi_P^{\pi}$)}
\item [ ]$=\varphi_P\circ ((\pi^{\prime}\circ \theta_{\pi^{\prime}}\circ f)\otimes (\phi_{P}\circ (f\otimes P)))\circ (\delta_{B}\otimes \tau)$ {\scriptsize (by (\ref{1-c1}) and (\ref{1-c2}))}
\item [ ]$= \varphi_P\circ ((\pi^{\prime}\circ \theta_{\pi^{\prime}})\otimes \phi_{P})\circ ((\delta_{B^{\prime}}\circ f)\otimes \tau)$ {\scriptsize (by the comonoid morphism condition for $f$)}
\item [ ]$=\tau\circ \phi_Q^{\pi}$ {\scriptsize (by (\ref{eq-gamma}))}
\end{itemize}

Then $(P,Q, \phi_P^{\pi}, \varphi_P^{\pi}, \phi_Q^{\pi}, \tau)$ is an object in $\;_{(\pi, \phi_{H}, \theta_{\pi})}{\sf Mod}$. Finally, it is obvious that, if $(h,l)$ is a morphism in $\;_{(\pi^{\prime}, \phi_{H^{\prime}}, \theta_{\pi^{\prime}})}{\sf Mod}$, then $(h,l)$ is a morphism in $\;_{(\pi, \phi_{H}, \theta_{\pi})}{\sf Mod}$. 
\end{proof}

\begin{remark}
\label{cd2}
If $(f,g)$ is an isomorphism defined between the generalized invertible 1-cocycles $(\pi: B\rightarrow H, \theta_{\pi})$ and $(\pi: B^{\prime}\rightarrow H^{\prime}, \theta_{\pi^{\prime}})$ with inverse $(f^{-1}, g^{-1})$, then the functor ${\sf M}_{(f,g)}$ is an isomorphism of categories with inverse ${\sf M}_{(f^{-1},g^{-1})}$.  For example, in the proof of Theorem \ref{EGIHT} we proved that, for all generalized invertible 1-cocycle $(\pi: B\rightarrow H, \theta_{\pi})$ , $(\pi, id_{H})$ is an isomorphism between the generalized invertible 1-cocycles $(\pi: B\rightarrow H, \theta_{\pi})$ , $(\pi, id_{H})$ and $(id_{H}:H_{\pi}\rightarrow H, \sigma_{\pi})$ where $\sigma_{\pi}=\pi\circ \theta_{\pi}\circ \pi^{-1}$ and the action is $\Gamma_{H}^{\sigma_{\pi}}$, i.e. the action is $\phi_{H}\circ (\pi^{-1}\otimes H)$ where $\phi_{H}$ is the action associated to $(\pi: B\rightarrow H, \theta_{\pi})$. Therefore, the functor 
$${\sf M}_{(\pi,id_{H})}:\;_{(id_{H}, \Gamma_{H}^{\sigma_{\pi}}, \sigma_{\pi})}{\sf Mod}\;\rightarrow \;_{(\pi, \phi_{H}, \theta_{\pi})}{\sf Mod}$$
is an isomorphism of categories with inverse 
$${\sf M}_{(\pi^{-1},id_{H})}:\;_{(\pi, \phi_{H}, \theta_{\pi})}{\sf Mod} \;\rightarrow \;_{(id_{H}, \Gamma_{H}^{\sigma_{\pi}}, \sigma_{\pi})}{\sf Mod}.$$
\end{remark}

\begin{theorem}
\label{th-mon-cat21}
Let ${\mathbb H}$ be a Hopf truss and let ${\sf E}({\mathbb H})$ be the invertible 1-cocycle induced by the  functor {\sf E} introduced in the proof of Theorem \ref{EGIHT}. There  exists a functor 
$${\sf G}_{{\mathbb H}}:\;_{{\mathbb H}}{\sf Mod}\;\rightarrow \;_{(id_{H}, \Gamma_{H_1}^{\sigma_{H}}, \sigma_{H})}{\sf Mod}$$ 
defined on objects by 
$${\sf G}_{{\mathbb H}}((M, \psi_{M}^1, \psi_{M}^2))=
(M,M, \widehat{\phi}_{M}=\Gamma_{M}^{\sigma_{H}}, \widehat{\varphi}_{M}=\psi_{M}^1,
\overline{\phi}_{M}=\psi_{M}^2, id_{M})$$
and on morphisms by ${\sf G}_{{\mathbb H}}(f)=(f,f)$.
\end{theorem}

\begin{proof}
By  assumption, $(M, \widehat{\varphi}_{M}=\psi_{M}^1)$ is a left $H_{1}$-module and $(M, \overline{\phi}_{M}=\psi_{M}^2)$ is a non-unital left $H_{2}$-module. On the other hand, by (\ref{GMH1}) we have that 
$$\widehat{\phi}_{M}\circ (H\otimes \widehat{\varphi}_{M})=\widehat{\varphi}_{M}\circ (\Gamma_{H_1}\otimes \widehat{\phi}_{M})\circ (H\otimes c_{H,H}\otimes M)\circ (\delta_{H}\otimes H\otimes M)$$
and then, (\ref{p-v}) holds. Also, 
\begin{itemize}
\item[ ]$\hspace{0.38cm} \widehat{\varphi}_{M}\circ (\sigma_{H}\ot \widehat{\phi}_{M})\circ (\delta_H\ot M)$
\item [ ]$=\psi_{M}^1\circ (\sigma_{H}\ot \Gamma_{M}^{\sigma_{H}})\circ (\delta_H\ot M) $ {\scriptsize (by definition of $ \widehat{\varphi}_{M}$ and $\widehat{\phi}_{M}$)}
\item [ ]$=\psi_{M}^1\circ (\sigma_{H}\ot (\psi_{M}^{1}\co ((\lambda_{H}\circ \sigma_{H})\ot \psi_{M}^{2})\co (\delta_{H}\ot M)))\circ (\delta_{H}\otimes M) $ {\scriptsize (by definition of $\Gamma_{M}^{\sigma_{H}}$)}
\item [ ]$=\psi_{M}^1\circ (((id_{H}\ast \lambda_{H}^{1})\circ \sigma_{H})\otimes \psi_{M}^2)\circ (\delta_{H}\otimes M) $ {\scriptsize (by the condition of left $H_{1}$-module of $(M, \psi_{M}^1)$, the }
\item[ ]$\hspace{0.38cm}$ {\scriptsize  coassociativity of $\delta_{H}$ and the condition of comonoid morphism for $\sigma_{H}$)}
\item [ ]$=\overline{\phi}_{M}$ {\scriptsize (by (\ref{antipode}), the counit properties, the condition of left $H_{1}$-module of $(M, \psi_{M}^1)$ and the definition of $ \overline{\phi}_{M}$).}
\end{itemize}
	
Finally, it is easy to show that, if $f$ is a morphism in $\;_{{\mathbb H}}{\sf Mod}$ between the objects $(M, \psi_{M}^1, \psi_{M}^2)$ and $(M^{\prime}, \psi_{M^{\prime}}^1, \psi_{M^{\prime}}^2)$, then the pair $(f,f)$ is a morphism in $\;_{(id_{H}, \Gamma_{H_1}^{\sigma_{H}}, \sigma_{H})}{\sf Mod}$ between ${\sf G}_{{\mathbb H}}((M, \psi_{M}^1, \psi_{M}^2))$ and ${\sf G}_{{\mathbb H}}((M^{\prime}, \psi_{M^{\prime}}^1, \psi_{M^{\prime}}^2))$.
\end{proof}

\begin{theorem}
\label{pHpi}
Let $(\pi: B\rightarrow H, \theta_{\pi})$ be a generalized invertible 1-cocycle. Then the categories $_{(\pi, \phi_{H}, \theta_{\pi})}{\sf Mod}$ and $\;_{{\Bbb H}_{\pi}}${\sf Mod} are equivalent.
\end{theorem}

\begin{proof} First of all we will prove that there  exists a functor 
$${\sf H}_{{\sf tr}}^{\pi}:\;_{(\pi, \phi_{H},\theta_{\pi})}{\sf Mod}\;\rightarrow \;_{{\mathbb H}_{\pi}}{\sf Mod}$$ 
defined on objects by 
$${\sf H}_{{\sf tr}}^{\pi}((M,N, \phi_M, \varphi_M,  \phi_N, \gamma))=
(M,\overline{\psi}_{M}^{1}=\varphi_M, \overline{\psi}_{M}^{2}=\gamma\circ \phi_N\circ (\pi^{-1}\otimes \gamma^{-1}))$$
and on morphisms by ${\sf H}_{{\sf tr}}^{\pi}((h,l))=h$. Indeed: By assumption, $(M,\overline{\psi}_{M}^{1}=\varphi_M)$ is a left $H$-module and, using the condition of non-unital left $B$-module of $N$, we obtain that $$(M, \overline{\psi}_{M}^{2}=\gamma\circ \phi_N\circ (\pi^{-1}\otimes \gamma^{-1}))$$ is a non-unital left $H_{\pi}$-module. Also, by (\ref{req-g2}), we have that the identity 
\begin{equation}
\label{pHpi1}
\Gamma_{M}^{\sigma_{\pi}}=\phi_{M}\circ (\pi^{-1}\otimes M)
\end{equation}
holds, where $\Gamma_{M}^{\sigma_{\pi}}=\overline{\psi}_{M}^{1}\circ ((\lambda_{H}\circ \sigma_{\pi})\otimes \overline{\psi}_{M}^{2})\circ (\delta_{H}\otimes M)$. Then, (\ref{mod-l1}) holds because:
\begin{itemize}
\item[ ]$\hspace{0.38cm} \overline{\psi}_{M}^{2}\circ (H\otimes \overline{\psi}_{M}^{1})$
\item [ ]$= \varphi_{M}\circ ((\pi\circ \theta_{\pi})\otimes\phi_{M})\circ ((\delta_{B}\circ \pi^{-1})\otimes \varphi_{M})$ {\scriptsize (by (\ref{req-g1}))}
\item [ ]$= \varphi_{M}\circ ((\pi\circ \theta_{\pi})\otimes (\varphi_{M}\circ (\phi_{H}\otimes \phi_{M})\circ (B\otimes c_{B,H}\otimes M)\circ (\delta_{B}\otimes H\otimes M)))\circ ((\delta_{B}\circ \pi^{-1})\otimes H\otimes M) $ 
\item[ ]$\hspace{0.38cm}${\scriptsize (by (\ref{p-v}))}
\item [ ]$=\varphi_{M}\circ  ((\mu_{H}\circ ((\pi\circ \theta_{\pi})\otimes \phi_{H})\circ (\delta_{B}\otimes \pi))\otimes \phi_{M})\circ (B\otimes c_{B,B}\otimes M)\circ ((\delta_{B}\circ \pi^{-1})\otimes \pi^{-1}\otimes M) $  
\item[ ]$\hspace{0.38cm}${\scriptsize (by the condition of left $H$-module for $M$, the coassociativity of $\delta_{B}$, the naturality of $c$ and the condition of}
\item[ ]$\hspace{0.38cm}${\scriptsize  isomorphism for $\pi$)}
\item [ ]$=\varphi_{M}\circ  ((\pi\circ \mu_{B})\otimes \phi_{M})\circ (B\otimes c_{B,B}\otimes M)\circ ((\delta_{B}\circ \pi^{-1})\otimes \pi^{-1}\otimes M) $ {\scriptsize (by (\ref{1-c}))}
\item [ ]$=\varphi_{M}\circ  (\mu_{H_{\pi}}\otimes (\phi_{M}\circ (\pi^{-1}\otimes M)))\circ (H\otimes c_{H,H}\otimes M)\circ (\delta_{H}\otimes H\otimes M) $ {\scriptsize (by the condition of coalgebra}
\item[ ]$\hspace{0.38cm}$  {\scriptsize  isomorphism for $\pi$ and the naturality of $c$)}
\item [ ]$=\overline{\psi}_{M}^{1}\circ (\mu_{H_{\pi}}\otimes \Gamma_{M}^{\sigma_{\pi}})\circ (H\otimes c_{H,H}\otimes M)\circ (\delta_{H}\otimes H\otimes M)$ {\scriptsize (by (\ref{pHpi1}))}
\end{itemize}
	
On the other hand, if $(h,l)$ is a morphisms in $\;_{(\pi, \phi_{H}, \theta_{\pi})}{\sf Mod}$ between \((M,N, \phi_M, \varphi_M,  \phi_N, \gamma)\) and 
\((M^{\prime},N^{\prime}, \phi_{M^{\prime}}, \varphi_{M^{\prime}},  \phi_{N^{\prime}}, \gamma^{\prime})\), then  we have that $h$ is a morphism in $\;_{{\mathbb H}_{\pi}}{\sf Mod}$  between $(M,\overline{\psi}_{M}^{1}, \overline{\psi}_{M}^{2})$ and $(M^{\prime},\overline{\psi}_{M^{\prime}}^{1}, \overline{\psi}_{M^{\prime}}^{2})$ because, using that $h$ is a morphism of  left $H$-modules, we have $h\circ \overline{\psi}_{M}^{1}=\overline{\psi}_{M^{\prime}}^{1}\circ (H\otimes h)$ and, by (\ref{fg-g}) and the condition of morphism of non-unital left $B$-modules for $l$, we have that $h\circ\overline{\psi}_{M}^{2}=\overline{\psi}_{M^{\prime}}^{2}\circ (H\otimes h)$.
	
Taking into account the functors ${\sf H}_{{\sf tr}}^{\pi}$, ${\sf G}_{{\mathbb H}_{\pi}}$ and ${\sf M}_{(\pi,id_{H})}$, it is easy to show that $${\sf H}_{{\sf br}}^{\pi}\circ ({\sf M}_{(\pi,id_{H})}\circ {\sf G}_{{\mathbb H}_{\pi}})	={\sf id}_{\;_{{\Bbb H}_{\pi}}{\sf Mod}}$$
and
$$(({\sf M}_{(\pi,id_{H})}\circ {\sf G}_{{\mathbb H}_{\pi}})\circ {\sf H}_{{\sf br}}^{\pi})((M,N, \phi_M, \varphi_M,  \phi_N, \gamma))=
(M,M,\phi_{M}, \varphi_{M},\overline{\phi}_{M}^{\pi}=\gamma\circ \phi_{N}\circ (A\otimes \gamma^{-1}), id_{M})$$
holds. Then, 
$$({\sf M}_{(\pi,id_{H})}\circ {\sf G}_{{\mathbb H}_{\pi}})\circ {\sf H}_{{\sf br}}^{\pi}\backsimeq {\sf id}_{_{(\pi, \phi_{H}, \theta_{\pi})}{\sf Mod}}$$
because $(id_{M}, \gamma)$ is an isomorphism in $_{(\pi, \phi_{H},\theta_{\pi})}{\sf Mod}$ between the objects $(M,N, \phi_M, \varphi_M,  \phi_N, \gamma)$ and $(M,M,\phi_{M}, \varphi_{M},\overline{\phi}_{M}^{\pi}, id_{M}).$
	
\end{proof}

\begin{remark}
When we particularize the previous results to Hopf braces we obtain the categorical equivalences obtained in \cite{VRBAMod}. More concretely, \cite[Theorem 2.26]{VRBAMod} is the Hopf brace version of Theorem \ref{pHpi1}.
\end{remark}

\section{The fundamental theorem of Hopf modules for Hopf trusses}

In this section we will assume that ${\sf C}$ admits equalizers. As a consequence every idempotent morphism in  $ {\sf C}$ splits, i.e., if $q:M\rightarrow M$ is a morphism in ${\sf C}$ such that $q=q\co q$, then there exists an object $I(q)$, called the image of $q$,  and morphisms $i:I(q)\rightarrow M$ and $p:M\rightarrow I(q)$ such that $q=i\co p$ and $p\co i =id_{I(q)}$. The morphisms $p$ and $i$ will be called a factorization of $q$. Note that $I(q)$, $p$ and $i$ are unique up to isomorphism.

Under the previous condition,  we will introduce the category ${\mathbb H}$-{\sf Hopf} of left  Hopf modules over a Hopf truss ${\mathbb H}$ and we will to obtain the Fundamental Theorem of Hopf modules for Hopf trusses that is a generalization of the one proved in \cite{RGON} for Hopf  brace.

\begin{definition}
Let  $D$ be a comonoid in {\sf C}. The pair
$(M,\rho_{M})$ is a left $D$-comodule if $M$ is an object in {\sf C} and $\rho_{M}:M\rightarrow D\ot M$ is a morphism
in {\sf C} satisfying $(\varepsilon\ot M)\co \rho_{M}=id_{M}$, $(D\ot \rho_{M})\co \rho_{M}=(\delta\ot M)\co \rho_{M}$. Given two left ${D}$-comodules $(M,\rho_{M})$
and $(N,\rho_{N})$, a morphism $f:M\rightarrow N$ in {\sf C} is a morphism of left  $D$-comodules if $(D\ot f)\co \rho_{M}=\rho_{N}\co f$.  Left $D$-comodules  with morphisms of left $D$-comodules  form a category which we denote by $\;^{{\sf D}}${\sf Comod}.
\end{definition}

\begin{definition}
\label{coinv}
{\rm Let $D$ be a comonoid such that there exits a comonoid morphism  $e:K\rightarrow D$. Let $(M,\rho_{M})$ be a left $D$-comodule. We define the subobject of coinvariants of $M$, denoted by $M^{coD}_{e}$, as the equalizer object of $\rho_{M}$ and $e\ot M$. Then, we have an equalizer diagram 
$$
\setlength{\unitlength}{1mm}
\begin{picture}(101.00,12.00)
\put(20,7.5){\vector(1,0){20}}
\put(46,8.5){\vector(1,0){27}}
\put(46,5.5){\vector(1,0){27}}
\put(13,7){\makebox(0,0){$M^{coD}_{e}$}}
\put(43,7){\makebox(0,0){$M$}}
\put(80,7){\makebox(0,0){$D\otimes M$}}
\put(27.5,10){\makebox(0,0){$j_{M}^e$}}
\put(60,11){\makebox(0,0){$\rho_{M}$}}
\put(60,3){\makebox(0,0){$e\ot M$}}
\end{picture}
$$
where $j_{M}^e$ denotes the equalizer (inclusion) morphism.

If $H$ is a Hopf monoid, then the unit $\eta_{H}$ is a comonoid morphism. Then, let $(M,\rho_{M})$ be a left $D$-comodule, we will denote the equalizer object of $\rho_{M}$ and $\eta_{H}\ot M$ by $M^{coD}$ and the equalizer  morphism by $j_{M}$.

}
\end{definition}

\begin{definition}
Let $B$ be a non-unital bimonoid. A non-unital left $B$-Hopf module over $B$ is a triple $(M,\varphi_{M}, \rho_{M})$ where $(M,\varphi_{M})$ is a non-unital left $B$-module, $(M, \rho_{M})$ is a left $B$-comodule and 
\begin{equation}
\label{HMOD}
 \rho_{M}\co \varphi_{M}=(\mu_{B}\ot \varphi_{M})\co (B\ot c_{B,B}\ot M)\co (\delta_B\ot \rho_{M})
\end{equation} 
holds. Non-unital left $B$-Hopf modules over $B$ with left linear and colinear morphisms form a category which we  denote by ${\sf B}$-{\sf Hopf-mod}.  
	
The definition for left $H$-Hopf modules over a Hopf monoid $H$ is similar changing non-unital left $H$-modules by  left $H$-modules. Then, in this case we will denote the category of $H$-Hopf modules over $H$ by ${\sf H}$-{\sf Hopf-Mod}
	
\end{definition}

\begin{remark}\label{rmk1}
Let $H$ be a Hopf monoid. It is easy to show that, if $(M,\varphi_{M}, \rho_{M})$ is a left $H$-Hopf module over $H$, then the endomorphism $q_{M}:M\rightarrow M$, defined by $$q_{M}=\varphi_{M}\co (\lambda_{H}\ot M)\co \rho_{M}$$ is idempotent and satisfies $\rho_{M}\co q_{M}=\eta_{H}\ot q_{M}$. Therefore, there exists a unique  morphism $$t_{M}:M\rightarrow M^{coH}$$ such that $t_{M}\co j_{M} =q_{M}.$ Let $I(q_{M})$ be the image of the idempotent morphism $q_{M}$ and let  $i_{M}:I(q_{M})\rightarrow M$ and $p_{M}:M\rightarrow I(q_{M})$ be the morphisms such that that $q_{M}=i_{M}\co p_{M}$ and $p_{M}\co i_{M} =id_{I(q_{M})}$. The morphism $$\omega_{M}=t_{M}\co i_{M}:I(q_{M})\rightarrow M^{coH}$$ is an isomorphism with inverse $\omega_{M}^{-1}=p_{M}\co j_{M}$. Moreover, $t_{M}\co \varphi_{M}=\varepsilon_{H}\ot t_{M}$ and, as a consequence, $(M^{co H}, t_{M})$ is the coequalizer of $\varphi_{M}$ and $\varepsilon_{H}\ot M$.

On the other hand, the object $H\ot M^{co H}$ is a left $H$-Hopf module with action $\varphi_{H\ot M^{co H}}=\mu_{H}\ot M^{co H}$ and coaction  $\rho_{H\ot M^{co H}}=\delta_{H}\ot M^{co H}$. The Fundamental Theorem of Hopf modules asserts that $H\ot M^{co H}$ is isomorphic to $M$ in the category ${\sf H}$-{\sf Hopf-Mod}. The isomorphism is defined by  $$\theta_{M}=\varphi_{M}\co (H\ot j_{M}): H\ot M^{co H}\rightarrow M$$ where $\theta_{M}^{-1}=(H\ot t_{M})\co \rho_{M}$.  In the same way as in the case of $M^{coH}$, if $X$ is an object in {\sf C}, then the tensor product $H\ot X$, with the action and coaction induced by the product and the coproduct of $H$, is a left $H$-Hopf module. Then,  there exists a functor ${\sf F}=H\ot -:{\sf C}\;\rightarrow   {\sf H}$-{\sf Hopf-Mod} called the induction functor. Also, for all $M\in {\sf H}$-{\sf Hopf-Mod}, the construction of  $M^{co H}$ is functorial. Thus, there exists a new functor ${\sf G}=(\;\;)^{co H}: {\sf H}$-{\sf Hopf-Mod}$\;\rightarrow {\sf C}$, called the functor  of coinvariants, such that ${\sf F}\dashv {\sf G}$. Moreover, ${\sf F}$ and ${\sf G}$ induce an equivalence between the categories ${\sf H}$-{\sf Hopf-Mod} and  ${\sf C}$.

Let us see now the definition of left Hopf module over a Hopf truss ${\mathbb H}$. 

\end{remark}

\begin{definition}
\label{deHmod}
Let ${\mathbb H}=(H_{1}, H_{2},  \sigma_{H})$ be a Hopf truss. A  left Hopf module over ${\mathbb H}$ (left ${\mathbb H}$-Hopf module) is a $4$-tuple $(M,\psi_{M}^1, \psi_{M}^2, \rho_{M})$ such that:
\begin{itemize}
\item[(i)] The triple $(M,\psi_{M}^1, \psi_{M}^2)$ is a left ${\mathbb H}$-module.
\item[(ii)] The triple $(M,\psi_{M}^1,\rho_{M})$ is a left $H_{1}$-Hopf module.
\item[(iii)] The triple $(M,\psi_{M}^2,\rho_{M})$ is a non-unital left $H_{2}$-Hopf module.
\item[(iv)] If $j_{M}$ is the equalizer morphism of $\rho_{M}$ and $\eta_{H}\ot M$, then the identity 
$$\psi_{M}^1\co (\sigma_{H}\ot j_{M})=\psi_{M}^2\co (H\ot j_{M})$$ 
holds.
\end{itemize}
		
A morphism of left Hopf modules over ${\mathbb H}$ is a morphism of left ${\mathbb H}$-modules and left $H$-comodules. 
		
Left Hopf modules over ${\mathbb H}$ with morphisms of left Hopf  modules  form a category which we denote by ${\mathbb H}$-{\sf Hopf-Mod}.  Note that this definition is a generalization to the Hopf truss setting of the notion of Hopf module over a Hopf brace introduced in \cite{RGON}.
\end{definition}

\begin{example}
\label{ej1}
Let $X$ be an object in {\sf C} and let ${\mathbb H}=(H_{1}, H_{2},  \sigma_{H})$ be a Hopf truss. Then, the $4$-tuple $$(H\ot X,\psi_{H\ot X}^1=\mu_{H}^1\otimes X, \psi_{H\ot X}^2=\mu_{H}^2\otimes X, \rho_{H\ot X}=\delta_{H}\otimes X)$$ is a left ${\mathbb H}$-Hopf  module. Indeed, by Example \ref{ex-mtr}, the triple $(H\ot X,\psi_{H\ot X}^1, \psi_{H\ot X}^2)$ is a left ${\mathbb H}$-module. Moreover, $(H\ot X, \psi_{H\ot X}^1, \rho_{H\ot X})$ is a left $H_{1}$-Hopf module and $(H\ot X, \psi_{H\ot X}^2, \rho_{H\ot X})$ is a non-unital left $H_{2}$-Hopf module. Finally, 
$$
\setlength{\unitlength}{1mm}
\begin{picture}(101.00,12.00)
\put(12,7.5){\vector(1,0){20}}
\put(46,8.5){\vector(1,0){27}}
\put(46,5.5){\vector(1,0){27}}
\put(9,7){\makebox(0,0){$X$}}
\put(39,7){\makebox(0,0){$H\ot X$}}
\put(85,7){\makebox(0,0){$H\otimes H\ot X$}}
\put(23.5,10){\makebox(0,0){$\eta_{H}\ot X$}}
\put(60,11){\makebox(0,0){$\delta_H\ot X$}}
\put(60,3){\makebox(0,0){$\eta_{H}\ot H\ot X$}}
\end{picture}
$$
is an equalizer diagram and this implies that $(H\ot X)^{co H_1}=X$ and $j_{H\ot X}=\eta_{H}\ot X$. Then, (iv) of Definition \ref{deHmod} holds, because by (\ref{cocycle}), we have that $\varphi_{H\ot X}\co (\sigma_{H}\ot j_{H\ot X})=\sigma_{H}\ot X=\psi_{H\ot X}\co (H\ot j_{H\ot X}).$ Note that, for $X=K$, we obtain that $$(H,\mu_{H}^1, \mu_{H}^2, \delta_H)$$ is an object in ${\mathbb H}$-{\sf Hopf-Mod} where $H^{coH_1}=K$.
		
On the other hand, $q_{H\ot X}=\varepsilon_H\ot \eta_{H}\ot X$ and, as a consequence, $I(q_{H\ot X})=X$, $p_{H\ot X}=\varepsilon_H\ot X$ and $i_{H\ot X}=\eta_{H}\ot X.$
		
Finally, it is obvious that, if $f:X\rightarrow Y$ is a morphism in {\sf C}, then $H\ot f$ is a morphism in ${\mathbb H}$-{\sf Hopf-Mod}.
\end{example}

\begin{example}\label{EXHT2-HM}
Recall from Example \ref{EXHT2-Mod} that given a Hopf monoid $D$, a comonoid morphism $q:D\to D$ satisfying (\ref{condBidemp}) and a left $D$-module $(M, \varphi_M)$ we obtain a left ${\mathbb D}_q$-module $(M, \varphi_M, \varphi_M^q)$. If, moreover, $(M, \varphi_M, \rho_M)$ is a left $D$-Hopf module, then it is easey to check that $(M, \varphi_M, \varphi_M^q, \rho_M)$ is a left ${\mathbb D}_q$-Hopf module. 

\end{example}

As a consequence Example \ref{ej1} we have the following results:

\begin{theorem}
\label{indF}
Let ${\mathbb H}$ be a Hopf truss. There exists a functor ${\sf F}=H\ot -:{\sf C}\;\rightarrow {\mathbb H}$-{\sf Hopf}, called the induction functor,  defined on objects by  ${\sf F}(X)=(H\ot X,\psi_{H\ot X}^1, \psi_{H\ot X}^2, \rho_{H\ot X})$ and on morphisms by ${\sf F}(f)=H\ot f.$
\end{theorem}

\begin{theorem} (Fundamental Theorem of Hopf modules)
\label{fun} Let ${\mathbb H}=(H_{1}, H_{2},  \sigma_{H})$ be a Hopf truss  and let $(M,\varphi_{M}, \psi_{M}, \rho_{M})$ be an object in ${\mathbb H}$-{\sf Hopf-Mod}. Then $(M,\psi_{M}^1, \psi_{M}^2, \rho_{M})$ and $F(M^{co H_1})$  are isomorphic in ${\mathbb H}$-{\sf Hopf-Mod}.
\end{theorem}

\begin{proof} Let $(M,\psi_{M}^1, \psi_{M}^2, \rho_{M})$ be an object in ${\mathbb H}$-{\sf Hopf}. By Theorem \ref{indF}, the $4$-tuple $$(H\ot M^{coH_1},\psi_{H\ot M^{coH_1}}^1, \psi_{H\ot M^{coH_1}}^2, \rho_{H\ot M^{coH_1}})$$  is a left ${\mathbb H}$-Hopf module. Let $\theta_{M}:H\ot M^{coH_1}\rightarrow M$ be the morphism defined by $\theta_{M}=\psi_{M}^1\co (H\ot j_{M})$. Then, by the general theory of Hopf modules exposed at the beginning of this section, $\theta_{M}$ is an isomorphism  of $H_{1}$-Hopf modules with inverse $\theta_{M}^{-1}=(H\ot t_{M})\co \rho_{M}$. Also, $\theta_{M}$ is an isomorphism  of non-unital $H_{2}$-Hopf modules because
\begin{itemize}
\item[ ]$\hspace{0.38cm} \psi_{M}^2\circ (H\otimes \theta_{M}) $
		
\item[ ]$= \psi_{M}^1\circ (\mu_{H}^{2}\otimes \Gamma_{M}^{\sigma_{H}})\circ (H\otimes c_{H,H}\otimes M)\circ (\delta_{H}\otimes H\otimes j_{M})$  {\scriptsize (by \ref{mod-l1})}
		
\item[ ]$=\psi_{M}^1\circ (\mu_{H}^{2}\otimes (\psi_{M}^{1}\co (H\ot \psi_{M}^{1})\co ((((\lambda_{H}\circ \sigma_{H})\otimes \sigma_{H})\circ \delta_{H})\ot M)))\circ (H\otimes c_{H,H}\otimes M)\circ (\delta_{H}\otimes H\otimes j_{M})$  
\item[ ]$\hspace{0.38cm}${\scriptsize (by (v) of definition \ref{deHmod})}
		
\item[ ]$=\psi_{M}^1\circ (\mu_{H}^{2}\otimes (\psi_{M}^{1}\co ((\lambda_{H}\ast id_{H})\circ \sigma_{H})\otimes M))\circ (H\otimes c_{H,H}\otimes M)\circ (\delta_{H}\otimes H\otimes j_{M}) $ {\scriptsize  (by the}
\item[ ]$\hspace{0.38cm}$ {\scriptsize  condition of comonoid morphism for $\sigma_{H}$ and the condition of left $H_{1}$-module for $M$)}
		
\item[ ]$=\psi_{M}^1\circ (\mu_{H}^2\ot j_{M})$ {\scriptsize  (by (\ref{antipode}), the naturality of $c$, the counit properties and the condition of left $H_{1}$-module for $M$)}
		
\item[ ]$=\theta_{M}\circ \psi_{H\ot M^{co H_{1}}}^2$ {\scriptsize  (by definition.)}
		
\end{itemize}
	
Therefore, $\theta_{M}$ is an isomorphism  in ${\mathbb H}$-{\sf Hopf-Mod} because, by (ii) of Definition \ref{deHmod}, the properties of $j_{M}$, the naturality of $c$ and the properties of $\eta_{H}$, we have that $\theta_{M}$ is a morphism of left $H$-comodules.
	
\end{proof}

In Theorem \ref{indF} we construct the induction functor ${\sf F}=H\ot -:{\sf C}\rightarrow {\mathbb H}$-{\sf Hopf-Mod} for a Hopf truss ${\mathbb H}=(H_{1}, H_{2},  \sigma_{H})$. As in the Hopf case,  for all $M\in {\mathbb H}$-{\sf Hopf-Mod}, the construction of the subobject of coinvariants $M^{co {\mathbb H}}:=M^{co H_{1}}$ is functorial. Thus, there exists a functor of coinvariants ${\sf W}=(\;\;)^{co {\mathbb H}}: {\mathbb H}$-{\sf Hopf-Mod}$\rightarrow {\sf C}$ defined on objects by $={\sf W}((M,\psi_{M}^1, \psi_{M}^2, \rho_{M}))=M^{co {\mathbb H}}$ where $M^{co {\mathbb H}}=M^{coH_1}$ and on  morphisms $f:M\rightarrow N$ by ${\sf W}(f)=f^{co{\mathbb H}}$, where $f^{co{\mathbb H}}:=f^{coH_1}$ is the unique morphism such that $j_{N}\co f^{coH_1}=f\circ j_{M}$. 

In the end  of this section we prove that there exists a categorical equivalence between {\sf C} and the category ${\mathbb H}$-{\sf Hopf-Mod} for a Hopf truss ${\mathbb H}$. 

\begin{theorem}
\label{prin2} Let ${\mathbb H}$ be a Hopf truss. The induction functor ${\sf F}=H\ot -:{\sf C}\rightarrow {\mathbb H}$-{\sf Hopf-Mod} is left adjoint of the functor of coinvariants ${\sf W}=(\;\;)^{co {\mathbb H}}: {\mathbb H}$-{\sf Hopf-Mod}$\rightarrow {\sf C}$  and they induce a categorical equivalence between ${\mathbb H}$-{\sf Hopf-Mod} and {\sf C}.
\end{theorem} 

\begin{proof} Let $X$ be an object in {\sf C}. Then the unit of the adjunction is $\alpha_{X}=id_{X}$ because as we proved in Example \ref{ej1}, ${\sf W}({\sf F}(X))=X$. For all $(M,\psi_{M}^1, \psi_{M}^2, \rho_{M})$  in ${\mathbb H}$-{\sf Hopf-Mod}, the counit is defined by $\beta_{M}=\theta_{M}$ where $\theta_{M}$ is the isomorphism introduced in the proof of Theorem \ref{fun}. The triangular identities hold, because they hold  for the adjunction between the categories {\sf C} and $H_{1}$-{\sf Hopf-Mod}.
\end{proof}

If we particularize the previous theorems to the case of Hopf braces, then we have the Fundamental Theorem of Hopf Modules and the associated categorical equivalence obtained in \cite{RGON}.

\begin{corollary}
Let ${\mathbb H}=(H_1,H_2)$ be a Hopf truss.  Then, the categories ${H_1}$-{\sf Hopf-Mod}  and ${\mathbb H}$-{\sf Hopf-Mod}  are equivalent.
\end{corollary}

\begin{proof}
By virtue of Remark \ref{rmk1} and Theorem \ref{prin2}, if ${\mathbb H} = (H_1, H_2)$ is a Hopf truss, then  the categories ${H_1}$-{\sf Hopf-Mod}  and ${\mathbb H}$-{\sf Hopf-Mod}  are equivalent because they are equivalent to {\sf C}. 
\end{proof}

\section*{Authors Contribution} Ram\'on Gonz\'alez Rodr\'{\i}guez and Ana Bel\'en  Rodr\'{\i}guez Raposo participated at all stages in the preparation of the manuscript.

\section*{Funding Declaration}
The  authors were supported by  Ministerio de Ciencia e Innovaci\'on of Spain. Agencia Estatal de Investigaci\'on. Uni\'on Europea - Fondo Europeo de Desarrollo Regional (FEDER). Grant PID2020-115155GB-I00: Homolog\'{\i}a, homotop\'{\i}a e invariantes categ\'oricos en grupos y \'algebras no asociativas.

\section*{Data Availability} Not applicable to this article.

\section*{Competing Interests Declaration} The authors declares no conflict of interest.


\begin{thebibliography}{99}

\bibitem{Abe} Abe E.: Hopf algebras, University Press, Cambridge, 1980.

\bibitem{AGV} Angiono I., Galindo C., Vendramin L.: Hopf braces and Yang-Baxter operators, Proc. Am. Math. Soc. 145, 1981-1995  (2017).

\bibitem{BRZ1} Brzezi\'nski T.: Trusses: between braces and rings, Trans. Am. Math. Soc. 372, 4149-4176 (2019).

\bibitem{VRBAMod} Fern\'andez Vilaboa J.M. , Gonz\'alez Rodr\'iguez R., Ramos P\'erez, B.,  Rodr\'iguez Raposo A.B., Modules for invertible 1-cocycles, Turkish J. Math 48, 248-266 (2023).

\bibitem{RGON} Gonz\'alez Rodr\'{\i}guez R.:  The fundamental theorem of Hopf modules for Hopf braces, Linear Multilinear Algebra 70, 5146-5156 (2022).

\bibitem{GV} Guarneri L., Vendramin L.: Skew braces and the Yang-Baxter equation, Math. Comput. 86, 2519-2534 (2017).

\bibitem{Larson-Sweedler}  Larson R.G.,  Sweedler M. E., An associative orthogonal bilinear form for Hopf
algebras, Amer. J. Math. 91, 75-93  (1969).

\bibitem{LW2} Li Z. and Wang S., Rota-Baxter systems of Hopf algebras and Hopf trusses, Preprint (2023), arXiv:2305.00482.

\bibitem{Christian}  Kassel C.: Quantum Groups,  Quantum groups. Graduate Texts in Mathematics, 155, Springer-Verlag, New York, 1995. 

\bibitem{Mac} Mac Lane S.:  Categories for the working mathematician. Second edition. Graduate Texts in Mathematics, 5, Springer-Verlag, New York, 1998. 

\bibitem{Sweedler}  Sweedler M.E., Hopf algebras, Benjamin, New York (1969).


\end{thebibliography}
\end{document}